\documentclass[11pt]{article}
\pdfoutput=1
\usepackage[utf8]{inputenc}
\usepackage{amsmath}
\usepackage{amssymb}
\usepackage{amsthm}
\usepackage{manfnt}
\usepackage{caption}

\counterwithin{figure}{section}
\newtheorem{theorem}{Theorem}[section]

\newtheorem{lemma}[theorem]{Lemma}
\newtheorem{corollary}[theorem]{Corollary}
\newtheorem{conjecture}[theorem]{Conjecture}

\theoremstyle{definition}

\newtheorem{remark}[theorem]{Remark}

\def\Block{\operatorname{Block}}
\def\Start{\operatorname{Start}}
\def\Skeleton{\operatorname{Skeleton}}
\def\Bone{\operatorname{Bone}}
\def\Stack{\operatorname{Stack}}
\def\Slice{\operatorname{Slice}}
\def\Comp{\operatorname{Comp}}

\makeatletter
\def\ps@pprintTitle{%
	\let\@oddhead\@empty
	\let\@evenhead\@empty
	\def\@oddfoot
	{\hbox to \textwidth%
		{\ifnopreprintline\relax\else
			\@myfooterfont%
			\ifx\@elsarticlemyfooteralign\@elsarticlemyfooteraligncenter%
			\hfil\@elsarticlemyfooter\hfil%
			\else%
			\ifx\@elsarticlemyfooteralign\@elsarticlemyfooteralignleft%
			\@elsarticlemyfooter\hfill{}%
			\else%
			\ifx\@elsarticlemyfooteralign\@elsarticlemyfooteralignright%
			{}\hfill\@elsarticlemyfooter%
			\else%
			Preprint submitted to \ifx\@journal\@empty%
			Elsevier%
			\else\@journal\fi\hfill\@date\fi%
			\fi%
			\fi%
			\fi%
		}%
	}%
	\let\@evenfoot\@oddfoot}
\makeatother

\usepackage{graphicx}
\usepackage[margin=1in]{geometry}
\usepackage{hyperref}
\title{Pull-Push Method: A new approach to Edge-Isoperimetric Problems}
\author{Sergei L. Bezrukov \and Nikola Kuzmanovski \and Jounglag Lim}
\date{\vspace{-5ex}}

\begin{document}

\maketitle
%auto-ignore
\begin{abstract}
	We prove a generalization of the Ahlswede-Cai local-global principle. A new technique to handle edge-isoperimetric
	problems is introduced which we call the pull-push method.
	Our main result includes all previously published results in this area as 
	special cases with the only exception of the
	edge-isoperimetric problem for grids.
	With this we partially answer a question of Harper on local-global principles.
	We also describe a strategy for further generalization of our results so that
	the case of grids would be covered, which would completely settle Harper's question.
\end{abstract}

%auto-ignore
\section{Introduction}

For a finite simple graph $G=(V,E)$, sets $A,B\subseteq V$ and  
integer $m\geq 0$ denote
\begin{align*}
    I_G(A,B) &= \{ \{u,v\}\in E \bigm | u\in A, v\in B  \},\\
    I_G(A) &= I_G(A,A),\\
    I_G(m) &= \max_{S\subseteq V, |S|=m} |I_G(S)|,\\
    \Theta (A) &= \{\{u,v\}\in E \bigm | u\in A, v\not \in A\},\\
    \Theta(m) &= \min_{S\subseteq V, |S|=m} |\Theta(m)|.
\end{align*}

In the sequel the index $G$ will be omitted whenever the graphs in
question are clear from the context. The following versions of the
edge-isoperimetric problem on graphs have been intensively studied in 
the literature:

\textit{Induced Edges Problem:} For a given $m\in \{1,\dots, |V|\}$ find a 
set $A\subseteq V$ such that $|A|=m$ and $I(|A|)=|I(A)|$.

\textit{Cut Edges Problem:} For a given $m\in \{1,\dots, |V|\}$ find a 
set $A\subseteq V$ such that $|A|=m$ and $\Theta(|A|)=|\Theta(A)|$.

Many authors have previously realized that these two problems are equivalent
for regular graphs, which follows from the next assertion.

\begin{lemma}
If $G=(V,E)$ is regular of degree $d$ and $A\subseteq V$ then 
\begin{align*}
    |\Theta(A)| + 2|I(A)| = d|A|.
\end{align*}
\end{lemma}

A set $A\subseteq V$ is called \textit{optimal} if $I_G(|A|)=|I_G(A)|$.
We say that $G=(V,E)$ admits \textit{nested solutions} if there exists a chain 
of optimal subsets $A_1\subset A_2 \subset \cdots \subset A_{|V|}$. In this
case we call the graph $G$ \textit{isoperimetric}.

A total order on a graph $G$ is a bijection 
$\eta_G: V\rightarrow \{1,\dots, |V|\}$. 
For positive integers $k,l$ with $k<l$, and $u,v \in V$, we define
\begin{align*}
    \eta_G[k,l] &= \eta^{-1}_{G}(\{k,\cdots , l\}),\\
    \eta_G[k] &= \eta^{-1}_{G}(\{1,\cdots , k\}),\\
    u<_{\eta_G} v &\mbox{ iff } \eta_G(u) < \eta_G(v).
\end{align*}
We call $\eta_G[k]$ \textit{initial segment} of size $k$ of the order $\eta_G$. 
For an isoperimetric graph $G=(V,E)$ and its chain of nested solutions
$\emptyset=A_0\subset A_1\subset A_2 \subset \cdots \subset A_{|V|}$ 
there is a natural total 
order $\eta_G$, which we call an \textit{optimal order} on $G$, defined by
$\eta^{-1}_G(i) = u_i$ for $\{u_i\}=A_i\setminus A_{i-1}$, 
$i\in \{1,\dots, |V|\}$. 
Note that an optimal order depends not only on $G$, but also on the chain
of nested solutions which is not unique, in general.  

We study edge-isoperimetric problems on Cartesian products of graphs.
For graphs $G$ and $H$ their \textit{Cartesian product} is
a graph $G\square H$ defined as follows:
\begin{align*}
    V_{G\square H} &= V_G \times V_H\\
    E_{G\square H} &= \{ ((v_G,v_H), (u_G,u_H) ) \bigm |
                     v_G=u_G \text{ and } (v_H,u_H)\in E_H, \text{ or } 
                     v_H=u_H \text{ and } (v_G,u_G)\in E_G\}.
\end{align*}
Denote $G^d=G\square \cdots \square G$ ($d$ times), where
$G^0$ is a simple graph with one vertex.

We define a \textit{Lexicographic order} of dimension $d$, $\mathcal{L}^d$ on $\mathbb{R}^d$,
such that for tuples of real numbers $x= (x_1, \dots, x_d)$ and
$y= (y_1, \dots , y_d)$ we say that $x <_{{\mathcal L}^d} y$ iff
$x_1=y_1,\dots, x_i=y_i$ and $x_{i+1}<y_{i+1}$ for some $i\in \{0,1,\dots, d-1\}$.

Suppose that $G_1,\dots, G_d$ are isoperimetric graphs with optimal orders
$\eta_1,\dots, \eta_d$, respectively, and let $G=G_1\square\cdots\square G_d$.
The following total order ${\mathcal L}^d_G$ on the Cartesian product of graphs, 
called \textit{Lexicographic order on} $G$ of dimension $d$, plays an important 
role in various extremal problems.
For $v=(v_1,\dots, v_d)\in V_G$ and $u=(u_1,\dots, u_d)\in V_G$, we write 
$v <_{{\mathcal L}^d_G} u$ iff
$(\eta_1(v_1),\dots, \eta_d(v_d)) <_{{\mathcal L}^d} 
(\eta_1(u_1),\dots, \eta_d(u_d))$.
The next theorem is one of the earliest results on edge-isoperimetric problems. 
\begin{theorem}[Harper \cite{Harp_Cube}, Bernstein \cite{B_Cube}, Hart \cite{Hart_Cube}]\label{cube}
%Fix an optimal ordering on $G=K_2$, there are only $2$ and both are optimal.
The order $\mathcal{L}_{G}^d$ is optimal for $G=K^d_2$.
\end{theorem}

A similar result was obtained later for the powers of larger cliques
and products of some other graphs, where the lexicographic order was proved
to be optimal. At the end of the twentieth century Ahlswede and Cai
established the optimality of the lexicographic order in the most general
form, called by them the \textit{local-global principle}, from which
Theorem \ref{cube} and many similar results follow.

\begin{theorem}[Ahlswede-Cai \cite{AC_LGB}]\label{AC_LGP}
Let $G_1,\dots, G_d$ be isoperimetric graphs and $G=G_1\square\cdots\square
G_d$. If the order ${\mathcal L}^2_{G_i \square G_j}$ is optimal for all 
$i,j$ with $1\leq i<j\leq d$, then ${\mathcal L}^d_{G}$ is optimal for $d\geq 3$.
\end{theorem}

\begin{remark}\label{subModFunc}
Actually, Theorem \ref{AC_LGP} was originally stated in a more general form
for sub-modular and super-modular functions over finite sets (see \cite{AC_LGB} and \cite{AC_SMF}), and the
functions $I_G$ and $\Theta_G$ belong to this category. Although our results 
can be also generalized for sub-modular and super-modular functions, 
we will only
be dealing with functions $I_G$ and $\Theta_G$ relevant to the
edge-isoperimetric problems.
\end{remark}

Despite many applications of Theorem \ref{AC_LGP}, there are graphs for whose Cartesian products the lexicographic order is not optimal. 
One of such graphs is a grid, i.e. the Cartesian product of paths, which was later
generalized to the product of arbitrary trees. 

\begin{theorem}[Bollobás-Leader \cite{BL_Grid}]\label{BL_Grid}
If $G$ is a path then $G^d$ has nested solutions for the Induced Edges problem.
\end{theorem}

\begin{theorem}[Ahlswede-Bezrukov \cite{AB_Trees}]\label{AB_Trees}
If $G_1,\dots, G_d$ are trees then $G_1\square \cdots \square G_d$ has
nested solutions for the Induced Edges problem.
\end{theorem}

Theorems \ref{BL_Grid} and \ref{AB_Trees} are proved combinatorially 
by induction on $d$. 
In \cite{BL_Grid} the authors also solved the Cut Edges problem, 
but they used calculus for this. The Cut Edges problem does not have 
nested solutions for the grid. Ahlswede and Bezrukov generalized 
Theorem \ref{BL_Grid} and gave a simpler proof based on a new approach. Later, 
some other graphs were found for which the optimal order is different
from the Lexicographic one.

\begin{theorem}[Bezrukov-Das-Elsässer \cite{BDE_Petersen}]\label{BDE_Petersen}
If $G$ is the Petersen graph then $G^d$ has nested solutions.
\end{theorem}

\begin{theorem}[Bezrukov-Das-Elsässer \cite{BDE_Petersen}]\label{BDE_PetersenCubes}
If $G_1$ is the Petersen graph and $G_2=K_2$ then
$G_1^{d_1} \square G_2^{d_2}$ has nested solutions.
\end{theorem}

\begin{theorem}[Carlson \cite{C_Tori}]\label{C_Cycle}
If $G$ is the cycle on $5$ vertices then $G^d$ has nested solutions.
\end{theorem}

The proof technique used in Theorems \ref{BDE_Petersen},
\ref{BDE_PetersenCubes} and \ref{C_Cycle} is also based on induction on
$d$, where the base case $d=2$ was considered specially. This approach can
also be considered as a local-global principle. The proof of the induction steps 
involved a large number of cases, subcases and subsubcases, along with a 
decent amount of computations, and is based on specific properties of the considered graphs.

Our result is the most general local-global principle. It is valid for
a variety of total orders, out of which the lexicographic order and the 
other orders appearing in theorems
\ref{BDE_Petersen}, \ref{BDE_PetersenCubes} and \ref{C_Cycle} are special
cases.  
The proof of our result is purely geometric and involves just a few cases or
computations. We use a new technique that we call the pull-push method. 
This technique is somewhat different from the one used to prove the earlier
theorems and does not depend on the structure of the involved graphs.
Harper in \cite{H_Book} asked if Theorem \ref{AC_LGP} can be extended to prove
Theorems \ref{BDE_Petersen} and \ref{BL_Grid}.
Our main result answers this affirmatively for Theorem \ref{BDE_Petersen}.
In the last section of the paper, 
we lay out a strategy on how to further generalize the main result to handle Theorem \ref{BL_Grid} as well.

The edge-isoperimetric problems have a lot of applications, some of which 
can be found in \cite{H_Book} and \cite{B2}.
The applications include the wirelength problem, the bisection width and the
edge congestion problem of graph embedding, modeling the brain, the 
cutwidth problem and graph partitioning. The graphs in Theorem 
\ref{BDE_Petersen} are called {\it folded Petersen networks} and have been 
studied in \cite{pet1,pet2,pet3,pet4, pet5} as a communication-efficient 
interconnection network topology for multiprocessors.
The graphs in Theorem \ref{BDE_PetersenCubes}
are called {\it folded Petersen cubes} and have been studied in
\cite{pet5, pet6}.

The paper has 7 sections. In the next section we introduce the necessary 
definitions to formulate the main result. In section 3 we present the 
geometric structure of the problem. In section 4 we go over some well-known results on compression which we use in the paper. 
The main result is proved in section 5. Section 6 is devoted to some corollaries of our main 
result and can be considered as a short survey of results in the area.
Concluding remarks and possible directions for generalizing our results
are put in section 7.

%auto-ignore
\section{\texorpdfstring{$\delta$} --sequences, partitions and statement of the main result}
Let ${\mathcal O}^d$ be a total order on $\mathbb{R}^d$ and let $G_1,\dots, G_d$ be isoperimetric graphs with optimal orders 
$\eta_1,\dots, \eta_d$. Consider $G=G_1\square \cdots \square G_d$ and
define the total order ${\mathcal O}^d_G$ on $G$ so that $v <_{{\mathcal O}^d_G} u$
for $v=(v_1,\dots, v_d)\in V_G$ and $u=(u_1,\dots, u_d)\in V_G$,
iff $(\eta_1(v_1),\dots, \eta_d(v_d)) <_{{\mathcal O}^d} 
(\eta_1(u_1),\dots, \eta_d(u_d))$.
We call ${\mathcal O}^d_G$ the \textit{induced by ${\mathcal O}^d$ order} on $G$.

Denote by $\mathfrak{S}_d$ the symmetric group of degree $d$, 
i.e., the set of all permutations on $\{1,\dots, d\}$.
For $\pi\in\mathfrak{S}_d$ we define the \textit{domination order} 
${\mathcal D}^{\pi, d}$ of dimension $d$ so that $x<_{{\mathcal D}^{\pi,d}} y$
for $x=(x_1,\dots, x_d)\in \mathbb{R}^d$ and
$y=(y_1,\dots, y_d)\in \mathbb{R}^d$ iff
$(x_{\pi(1)},\dots,x_{\pi(d)})<_{{\mathcal L}^d} (y_{\pi(1)},\dots,y_{\pi(d)})$.
Respectively, for $G=G_1\square \cdots \square G_d$ we obtain the domination
order ${\mathcal D}^{\pi, d}_G$ on $G$ induced by ${\mathcal D}^{\pi, d}$.
For example, ${\mathcal D}^{\text{id},d} = {\mathcal L}^d$,
where $\text{id}$ is the identity permutation on $\{1,\dots, d\}$.
If $\pi(i) = (d-i+1)$ then ${\mathcal D}^{\pi,d}$ is the 
\textit{colexicographic order}.
For $d=2$ we have just two domination orders, the
lexicographic and colexicographic ones. A geometric interpretation of
these orders is given in Figure \ref{dom2D}.
For $d=3$ there are $6$ dominations orders including
the lexicographic and colexicographic ones.
A geometric interpretation of these orders is given in Figure \ref{dom3D}.

\begin{figure}
	\centering
	\includegraphics[width=\textwidth]{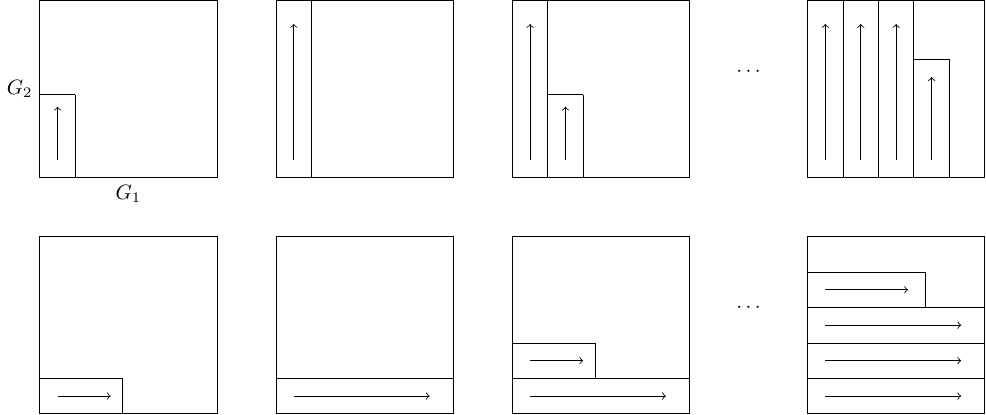}
	\caption{Geometric interpretation of lexicographic (top) and colexicographic (bottom) orders in two dimensions}\label{dom2D}
\end{figure}

\begin{figure}
	\centering
	\includegraphics[width=\textwidth]{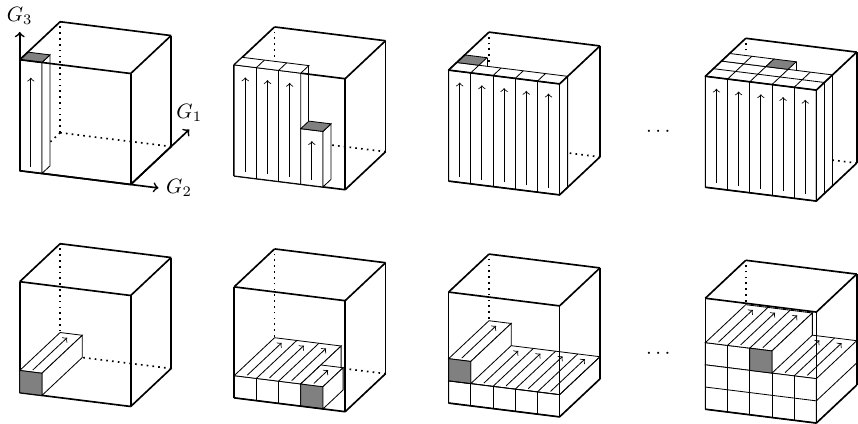}
	\caption{Geometric interpretation of lexicographic (top) and colexicographic (bottom) orders in three dimensions}\label{dom3D}
\end{figure}

%\begin{corollary} (of Theorem \ref{AC)LGB})\label{AC_LGP_CoLex}
%Let $G_1,\dots, G_d$ be isoperimetric graphs.
%If for all $i,j\in \{1,\dots, d\}$ with $i>j$ the order 
%${\mathcal C}^2_{G_i \square G_j}$ is optimal, 
%then ${\mathcal C}^d_{G}$ is optimal for $d\geq 3$.
%\end{corollary}
%\begin{proof}
%Denote $H = G_d\square \cdots \square G_1$ and define 
%$\psi: G\rightarrow H$ such that
%\begin{align*}
%    \psi((g_1,\dots, g_d)) = (g_d,\dots, g_1).
%\end{align*}
%It is easily seen that $\psi$ is a graph isomorphism.
%For $x,y\in V_G$ one has $x<_{{\mathcal C}^d_G} y$ iff
%$\psi(x)<_{{\mathcal L}^d_H} \psi(y)$.
%Since ${\mathcal L}^d_H$ is optimal for $H$ we conclude
%${\mathcal C}^d_G$ is optimal for $G$.
%\end{proof}

\begin{corollary} (of Theorem \ref{AC_LGP})\label{AC_LGP_Dom}
Let $G_1,\dots, G_d$ be isoperimetric graphs, $G = G_1\square \cdots \square
G_d$, and $\pi \in \mathfrak{S}_d$.
If for all $i,j \in \{1,\dots ,d\}$ with $i< j$ the order
${\mathcal L}^2_{G_{\pi(i)}\square G_{\pi(j)}}$ 
is optimal, then ${\mathcal D}^{\pi,d}_G$ is optimal for $d\geq 3$.
\end{corollary}
\begin{proof}
Denote $H=G_{\pi(1)}\square\cdots\square G_{\pi(d)}$ and define
$\psi: G\rightarrow H$ such that
\begin{align*}
    \psi((v_1,\dots, v_d)) = (v_{\pi(1)},\dots,v_{\pi(d)}).
\end{align*}
It is easily seen that $\psi$ is a graph isomorphism.
For $x,y\in V_G$ one has $x<_{{\mathcal D}^{\pi,d}_G} y$ iff
$\psi(x)<_{{\mathcal L}^d_H} \psi(y)$.
Since ${\mathcal L}^d_H$ is optimal for $H$ we conclude
${\mathcal D}^{\pi,d}_G$ is optimal for $G$.
\end{proof}

For a graph $G=(V,E)$ and integer $m\in \{1,\dots, |V|\}$ denote
\begin{align*}
    \delta_G(m) = I_G(m) - I_G(m-1)
\end{align*}
with $\delta_G(1)=0$. If $G$ is isoperimetric with optimal order ${\mathcal O}_G$ 
and $v\in V$ denote
\begin{align*}
    \Delta_G(v) = \delta_G( {\mathcal O}_G(v))
\end{align*}

\begin{lemma}[Bezrukov \cite{B_equivalence}]\label{deltaSeqBound}
If $G=(V,E)$ is isoperimetric then $\delta(i+1) - \delta(i) \leq 1$ for
$i=1,\dots,|V|-1$. 
\end{lemma}

Let $G=(V,E)$ be isoperimetric with optimal order ${\mathcal O}_G$.
For integers $a,b\in \{1,\dots, |V|\}$ with $a<b$
denote by $\delta_{G}[a,b] = (\delta(a),\dots, \delta(b))$  , a \textit{monotonic segment}, where
\begin{enumerate}
    \item For all $i\in \{a,\dots, b-1\}$ we have $\delta(i+1)-\delta(i)=1$.
    \item If $a>1$ then $\delta(a)-\delta(a-1) < 1$.
    \item If $b< |V|$ then $\delta(b+1)-\delta(b) < 1$.
\end{enumerate}
In other words, a monotonic segment is a longest increasing sequence of 
$\delta$-values.
We say that ${\mathcal O}_G[a,b] \subseteq V$ is a \textit{monotonic set} from $a$ to $b$.
For two monotonic sets ${\mathcal O}_G[a_1,b_1]$ and ${\mathcal O}_G[a_2,b_2]$,
we write ${\mathcal O}_G[a_1,b_1] <_{{\mathcal O}_G} {\mathcal O}_G[a_2,b_2]$ 
iff $b_1<a_2$.
It is easily seen that $V$ is uniquely partitioned into monotonic sets,
hence, $\delta_G$ - into monotonic segments.
%Similarly, $G$ can be partitioned into monotonic sets 
%${\mathcal O}_G[a_1,b_1] <_{{\mathcal O}_G} \cdots <_{{\mathcal O}_G} {\mathcal O}_G[a_k,b_k]$
%and 
We call such a partition the \textit{standard monotonic partition} of $G$, and we denote it by $\mathfrak{M}_G=\{{\mathcal O}_G[a_1,b_1],\dots,{\mathcal O}_G[a_k,b_k]\}$. 
Here are a few examples, where monotonic segments are underlined:
\begin{align*}
    \delta_{K_n} &= (\underline{0,1,2,3,\dots,n-1}),\\
    \delta_{P_n} &= (\underline{0,1},\underline{1},\underline{1},\dots, 
                    \underline{1}),\\
    \delta_{\text{Petersen}} &= (\underline{0,1},\underline{1},\underline{1,2},
                    \underline{1,2},\underline{2},\underline{2,3}).
\end{align*}
Thus, $|\mathfrak{M}_{K_n}|=1$, $|\mathfrak{M}_{P_n}|=n-1$, and 
$|\mathfrak{M}_{\text{Petersen}}|= 6$.
It turns out that $\mathfrak{M}_G$ has interesting properties.

\begin{theorem}[Bezrukov-Bulatovic-Kuzmanovski \cite{BBK}]\label{standardPartition}
Let $G$ be isoperimetric and consider its standard monotonic partition
$\mathfrak{M}_G =
\{{\mathcal O}_G[a_1,b_1] <_{{\mathcal O}_G} \cdots <_{{\mathcal O}_G} {\mathcal
O}_G[a_k,b_k]\}$. For
$i\in \{1,\dots, k\}$ one has
\begin{enumerate}
    \item The graph $({\mathcal O}_G[a_i,b_i], I_G({\mathcal O}_G[a_i,b_i]))$ is a clique,
    and hence it is isoperimetric with induced order ${\mathcal O}_i$,
    such that $u <_{{\mathcal O}_i} v$ for $u,v\in {\mathcal O}_G[a_i,b_i]$ 
    iff $u < _{{\mathcal O}_G} v$.
    \item For all $v\in {\mathcal O}_G[a_i,b_i]$ we have 
    $|I({\mathcal O}_G[a_1,b_{i-1}], \{v\})| = \delta (a_i)$.
\end{enumerate}
\end{theorem}

We extend the concept of monotonic partitions to more general partitions
$\mathfrak{P}_G = 
\{{\mathcal O}_G[a_1,b_1]<_{{\mathcal O}_G}\cdots <_{{\mathcal O}_G}{\mathcal
O}_G[a_k,b_k]\}$ of $V$, where ${\mathcal O}_G[a_i,b_i]$ are not necessarily 
monotonic sets. We say that $\mathfrak{P}_G$ is an 
\textit{isoperimetric partition} if
\begin{enumerate}
    \item The graph $({\mathcal O}_G[a_i,b_i], I_G({\mathcal O}_G[a_i,b_i]))$ 
is isoperimetric with induced order ${\mathcal O}_i$,
    such that $u <_{{\mathcal O}_G} v$ for $u,v\in {\mathcal O}_G[a_i,b_i]$ 
    iff $u < _{{\mathcal O}_i} v$.
    \item For every $v\in {\mathcal O}_G[a_i,b_i]$ it holds 
    $|I({\mathcal O}_G[a_1,b_{i-1}], \{v\})| = \delta (a_i) $.
\end{enumerate}

For example, for the Petersen graph we can partition the $\delta$-sequence in two parts
$$\delta_{\text{Petersen}} = (\underline{0,1,1,1,2},
\underline{1,2,2,2,3}).$$
Each of the parts induces a cycle of length $5$ studied in \cite{C_Tori}.

\begin{lemma}\label{deltaInPartition}
Let $G=(V,E)$ be an isoperimetric graph with an isoperimetric partition
$\mathfrak{P}_G = 
\{{\mathcal O}_G[a_1,b_1] <_{{\mathcal O}_G} \cdots <_{{\mathcal O}_G} {\mathcal O}_G[a_k,b_k]\} $.
Then for the graph $H_i=(\mathcal{O}_G[a_i,b_i], I_G(\mathcal{O}_G[a_i,b_i]))$,
$i=1,\dots,k$,
and $x,y\in \mathcal{O}_G[a_i,b_i]$ with $y<_{\mathcal{O}_G} x$ it holds that
\begin{align*}
    \Delta_G(x)- \Delta_G(y) = \Delta_{H_i}(x)-\Delta_{H_i}(y).
\end{align*}
\end{lemma}
\begin{proof}
Indeed,
    $\Delta_G(x)- \Delta_G(y) 
    = \Delta_G(a_i) +\Delta_{H_i}(x)-\Delta_G(a_i)-\Delta_{H_i}(y)
    = \Delta_{H_i}(x)-\Delta_{H_i}(y)$.
\end{proof}

We call the first and last vertex of ${\mathcal O}[a,b]\in \mathfrak{P}_G$ the \textit{start} and \textit{end} of ${\mathcal O}[a,b]$.
Further denote by $\mathfrak{T}_G=\{{\mathcal O}^{-1}_G(a_1),\dots, {\mathcal O}^{-1}_G(a_k)\}$ 
the \textit{start set} of the partition $\mathfrak{P}_G$.
We say that $\mathfrak{P}_G$ is \textit{non-decreasing} if for every 
$i\in \{1,\dots, k\}$ the sequence
$\delta_{({\mathcal O}_G[a_i,b_i], I_G({\mathcal O}_G[a_i,b_i]))}$ is non-decreasing.
We say that $\mathfrak{P}_G$ is \textit{regular} if
\begin{align*}
    \delta_{({\mathcal O}_G[a_1,b_1], I_G({\mathcal O}_G[a_1,b_1]))} = \delta_{({\mathcal O}_G[a_k,b_k], I_G({\mathcal O}_G[a_k,b_k]))}
\end{align*}
%Note that for any isoperimetric graph $G$,
%the atomic partition $\mathfrak{A}_G$ is an isoperimetric, non-decreasing 
%and regular partition.
Note that for any isoperimetric graph $G$, the standard monotonic partition 
$\mathfrak{M}_G$ is an isoperimetric and non-decreasing partition.
However, $\mathfrak{M}_G$ is not always regular, as the next example shows.
Consider the graph $G$ which is the union of two disjoint cliques $K_5$ and $K_4$. 
Then
\begin{align*}
    \delta_{G}=(0,1,2,3,4,0,1,2,3).
\end{align*}

%In the sequel we will be dealing with higher dimensional objects formed 
%from the isoperimetric partitions. 
For $G=G_1\square \cdots \square G_d$ and nonempty subset 
$S=\{i_1,\dots, i_k\}\subseteq \{1,\dots, d\}$
we define the \textit{subproduct} of $G$ of dimension $k$
as $G_{S}=G_{i_1}\square \cdots \square G_{i_k}$.
Let ${\mathcal O}_G$ and ${\mathcal O}_{G_S}$ be total orders on $G$ and $G_S$,
respectively.
We say that ${\mathcal O}_G$ is \textit{consistent} with ${\mathcal
O}_{G_S}$ if $x <_{{\mathcal O}_G} y$ for $x=(x_1,\dots, x_d)\in V_G$ and 
$y=(y_1,\dots, y_d)\in V_G$ with $x_j = y_j$ for $j\not\in S$ implies 
$(x_{i_1},\dots,x_{i_k}) <_{{\mathcal O}_{G_S}} (y_{i_1},\dots,y_{i_k})$.

Suppose $G_i=(V_i,E_i)$ for $i=1,\dots,d$ is an isoperimetric graph
with an optimal order ${\mathcal O}_{G_i}$, and let $\mathfrak{P}_{G_i}$ be its
isoperimetric partition with the start set $\mathfrak{T}_{G_i}$.
Define a \textit{block} of $G$ to be an element of
\begin{align*}
    \{ Z_1\times \cdots \times Z_d \bigm | Z_1\in \mathfrak{P}_{G_i},\dots , Z_d\in \mathfrak{P}_{G_d} \}
\end{align*}
and a \textit{start} of $G$ to be an element of
\begin{align*}
    \{ (z_1,\dots, z_d) \bigm | z_1\in \mathfrak{T}_{G_1},\dots , z_d\in \mathfrak{T}_{G_d}    \}.
\end{align*}
Note that blocks of $G$ and starts of $G$ are in a bijective correspondence.
We say that the partitions $\mathfrak{P}_{G_1},\dots, \mathfrak{P}_{G_d}$
compose a \textit{domination collection} if for each block 
$B=Z_1\times \cdots \times Z_d$ of $G$ one has:
\begin{enumerate}
    \item For each nonempty $S=\{i_1<\cdots < i_k\} \subseteq \{1,\dots, d\}$ 
there is an optimal domination order ${\mathcal D}^{\pi_S, k}_{H_S}$
on graph $H_S = (Z_{i_1}\times \cdots \times Z_{i_k}, I_G(Z_{i_1}\times \cdots
\times Z_{i_k}) )$.
    \item For any nonempty sets $S_1 = \{i_1<\cdots < i_{k_1}\} 
\subseteq \{1,\dots, d\}$ and $S_2 = \{i_1<\cdots < i_{k_2}\} 
\subseteq \{1,\dots, d\}$ with $S_1\subset S_2$,
the order ${\mathcal D}^{\pi_{S_2}, k_2}_{H_{S_2}}$ is consistent with 
${\mathcal D}^{\pi_{S_1}, k_1}_{H_{S_1}}$. 
\end{enumerate}

We can obtain a domination collection by defining domination orders
for each block of $G$.
Note that for any $S\subset \{1,\dots, d\}$ 
the vertices of $H_S$ form a block in the subproduct $G_S$.
Also note that $\mathfrak{P}_{G_{i_1}},\dots, \mathfrak{P}_{G_{i_k}}$
is a domination collection on $G_S$.
Hence, each of $2^d-1$ subproducts has a start set, blocks, and domination 
collection. For brevity, we denote the domination order 
on a block $B$ of some $k$-dimensional subproduct by $\mathcal{D}_B$. 

Now, we introduce a new total order for which we prove a local-global
principle in the next sections.
For nonempty $S=\{i_1<\dots < i_k\} \subseteq \{1,\dots, d\}$ define 
the \textit{block lexicographic order} ${\mathcal{BL}}_{G_S}^{k}$ of dimension 
$k$ on $G_S$ such that for $u,v\in V_G$ we have $u <_{\mathcal{BL}_{G_S}^{k}} v$ 
iff
\begin{enumerate}
    \item If $u$ and $v$ are in the same block $B$, then $u <_{{\mathcal D}_B} v$.
    \item If $u$ and $v$ are in different blocks, say $B_u$ and $B_v$, with respective starts $z_u$ and $z_v$, then $z_u <_{{\mathcal L}^k_G} z_v$.
\end{enumerate}
We abbreviate ${\mathcal{BL}}_{G_{\{1,\dots, d\}}}^{d}$ to ${\mathcal{BL}}_{G}^{d}$.
Just one more definition is needed to state our main result below.
Suppose that for $d\geq 3$ and $i=1,\dots,d$ we have an isoperimetric graph 
$G_i=(V_i,E_i)$ with optimal order ${\mathcal O}_{G_i}$ and isoperimetric partition 
$\mathfrak{P}_{G_i}=\{\mathcal{O}_{G_i}[a_{i,1},b_{i,1}] <_{\mathcal{O}_{G_i}} 
\cdots <_{\mathcal{O}_{G_i}} \mathcal{O}_{G_i}[a_{i,n_i}, b_{i,n_i}]\}$.
We say that $\mathfrak{P}_{G_1} ,\dots, \mathfrak{P}_{G_d}$ is 
a \textit{regular domination collection} if the following hold:
\begin{enumerate}
    \item The partition $\mathfrak{P}_{G_i}$ is regular for $i=2,\dots,d-1$.
    \item With
    \begin{align*}
        B_{1}&={O}_{G_2}[a_{2,1},b_{2,1}]\times \cdots \times {O}_{G_{d-1}}[a_{d-1,1},b_{d-1,1}],\\
        B_{2}&={O}_{G_{2}}[a_{2,n_2},b_{2,n_2}]\times \cdots \times {O}_{G_{d-1}}[a_{d-1,n_{d-1}},b_{d-1,n_{d-1}}].
    \end{align*}
the domination order on $B_1$ is the same as the domination order on 
$B_2$. That is, if $\mathcal{D}_{B_1}$ is induced by a permutation $\pi_1$ and 
$\mathcal{D}_{B_2}$ is induced by a permutation $\pi_2$, then $\pi_1=\pi_2$.
\end{enumerate}

\begin{theorem}\label{mainResult}
Let $G_1,\dots, G_d$ be isoperimetric graphs and let their corresponding 
isoperimetric partitions be $\mathfrak{P}_{G_1},\dots ,\mathfrak{P}_{G_d}$.
Denote $G=G_1\square \cdots \square G_d$ and suppose that the following hold:
\begin{enumerate}
    \item For $i=1,\dots,d-1$ the partition $\mathfrak{P}_{G_i}$ is 
non-decreasing.
    \item The collection of partitions $\mathfrak{P}_{G_1},\dots, \mathfrak{P}_{G_d}$ is a regular domination collection.
\end{enumerate}
If for $i,j\in \{1,\dots, d\}$ with $i<j$ the order 
$\mathcal{BL}_{G_i\square G_j}^{2}$ is optimal, then the order 
$\mathcal{BL}_G^d$ is optimal for $d\geq 3$.
\end{theorem}

%auto-ignore
\section{Geometry of the problem}
The objects introduced here play a key role in understanding transformations 
and proof techniques used in the paper. Throughout this section we assume that
$G_1,\dots,G_d$ are isoperimetric graphs with domination collection
$\mathfrak{P}_{G_1},\dots , \mathfrak{P}_{G_d}$ and start sets
$\mathfrak{T}_{G_1},\dots , \mathfrak{T}_{G_d}$.
Denote $G=G_1\square \cdots \square G_d =(V,E)$.
We can view $G$ as a $d$-dimensional rectangular body where the vectors 
with integer coordinates (see Figure \ref{Order of the cubes in C_4^3}) correspond to
the vertices of $G$, and the vector coordinates along each coordinate axis are ordered 
according to the isoperimetric order on $G_i$.

For $u\in V$ denote by $\Block_G(u)$ the $d$-dimensional block containing 
$u$. Since the blocks of $G$ partition $V$, for every $u\in V$ its containing
block is defined uniquely. Denote by $\Start_G(B)$ the start vertex of block
$B$ and by $\Start_G(u)$ the start vertex of the block $\Block_G(u)$.
Figure \ref{Order of the cubes in C_4^3} shows some blocks and their starts.

\begin{figure}
	\centering
	\includegraphics[width=\textwidth]{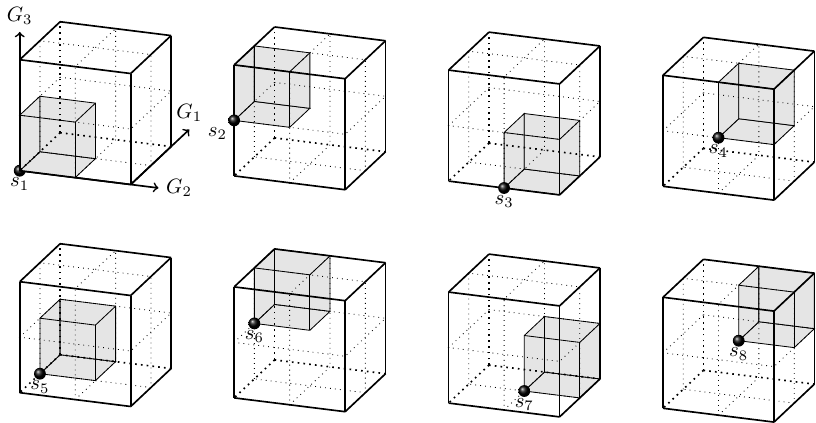}
	\caption{Ordering of blocks of $G$ for $d=3$,
		$|\mathfrak{P}_{G_1}|=|\mathfrak{P}_{G_2}|=|\mathfrak{P}_{G_3}|=2$,
		and $\mathfrak{T}_{G_1}\times \mathfrak{T}_{G_2} \times \mathfrak{T}_{G_3}= \{s_1,s_2,s_3,s_4,s_5,s_6,s_7,s_8\}$
		listed in increasing order.}
	\label{Order of the cubes in C_4^3}
\end{figure}

For $i\in \{1,\dots,d\}$ define the $i$-th \textit{bone} and the
\textit{skeleton} of a block $B=Z_1\times \cdots \times Z_d$ as
\begin{eqnarray*}
    \Bone_G(B,i) &=& \left( \prod_{j=1}^{i-1}\{\Start_{G_j}(Z_j)\}\right)
    \times Z_i \times \left( \prod_{j=i+1}^{d}\{\Start_{G_j}(Z_j)\}\right)\\
    \Skeleton(B) &=& \bigcup_{i=1}^d \Bone_G(B,i).
\end{eqnarray*}
Figure \ref{bones and skeleton} shows a visualization of the bones and skeleton of a block.

\begin{figure}
	\centering
	\includegraphics[width=\textwidth]{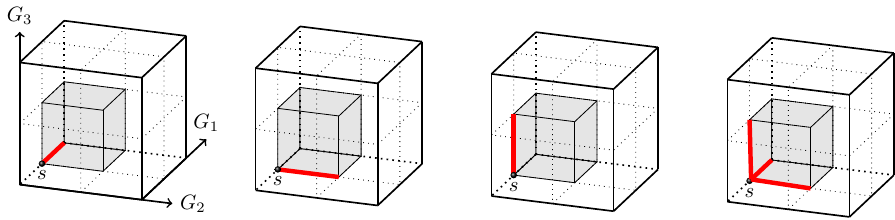}
	\caption{Bones and the skeleton of the block with start $s$.}
	\label{bones and skeleton}
\end{figure}

If $C=Y_1\times \cdots \times Y_d$ is some block other than $B$ with
$Y_i=Z_i$, we say that blocks $B$ and $C$ \textit{share} the $i$-th bone in the \textit{product decomposition} of $C$ and $B$.
Figure \ref{bone sharing} shows examples of bone sharing.

\begin{figure}
	\centering
	\includegraphics[width=0.8\textwidth]{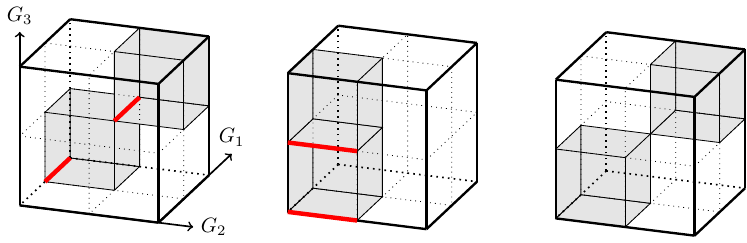}
	\caption{Blocks sharing 1st bone, 2nd bone, and no bone.}
	\label{bone sharing}
\end{figure}

For $i\in \{1,\dots, d\}$, $\sigma\in\mathfrak{T}_{G_1}\times\cdots 
\times \mathfrak{T}_{G_{i-1}}$, $\tau \in \mathfrak{T}_{G_{i+1}}\times 
\cdots \times \mathfrak{T}_{G_{d}}$, and $s\in \mathfrak{T}_{G_i}$ we consider
the vertices of the form $(\sigma,s,\tau)\in V$. 
Then define the \textit{stack} in direction $i$ at $\alpha=(\sigma,i,\tau)$ as
\begin{align*}
    \Stack_G(\alpha) = \bigcup_{s\in \mathfrak{T}_{G_i}} 
    \Block((\sigma, s, \tau)).
\end{align*}
Some stacks are visualized in Figure \ref{cube stacks C_alpha and beta}.
For the left stack one has $\sigma = (s_1,s_2)$ where $s_1$ is the fourth start of $G_1$ and $s_2$ is the first start of $G_2$,
and there is no $\tau$.
For the other stack $\tau = (s_3,s_4)$ where $s_3$ is the second start of $G_2$ and $s_4$ is the first start of $G_3$, and there is no $\sigma$.

\begin{figure}
	\centering
	\includegraphics[width=0.3\textwidth]{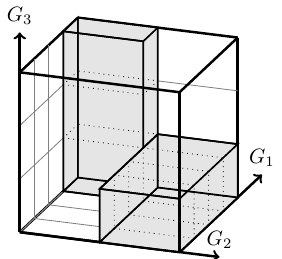}
	\caption{Visualization of stacks of $G$ for $d=3$,
		the left stack is in direction $3$, the other one is in direction $1$.}
	\label{cube stacks C_alpha and beta}
\end{figure}

The last objects we will need are called \textit{slices}.
For $i\in \{1,\dots, |\mathfrak{T}_{G_1}|\}$ denote by $s_i$ the
$i$-th start of $G_1$.
The $i$-th \textit{slice} of $G$ is defined as the union of all 
blocks $B$ whose first coordinate of $\Start_G(B)$ is $s_i$,
and denoted by $\Slice_G(i)$.
In other terms a slice is the union of all blocks that share the 1st bone.
Some slices are shown in Figure \ref{block slices}.

\begin{figure}
	\centering
	\includegraphics[width=0.55\textwidth]{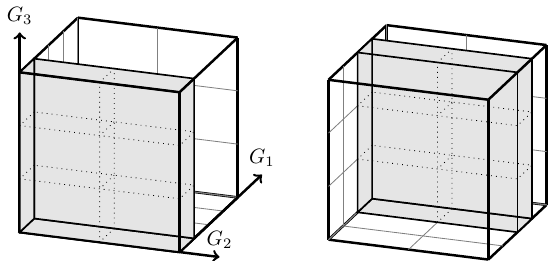}
	\caption{ $\Slice_G(1)$ and  $\Slice_G(3)$.}
	\label{block slices}
\end{figure}

For blocks $B_1$ and $B_2$ of $G$ we say that $B_1 <_{\mathcal{BL}_G^d} B_2$ 
iff $\Start_G(B_1) <_{\mathcal{BL}_G^d} \Start_G(B_2)$.
This way we obtain a total order on the set of blocks of $G$, which is
illustrated in Figure \ref{Order of the cubes in C_4^3}

Since a stack is a disjoint union of blocks,
all the blocks of a stack become totally ordered.
For an example, where blocks are ordered according to the indices of their starts, see the left part of Figure \ref{Ordering of blocks in a stack and ordering of stacks in a slice}.
For stacks $\Stack_G(\alpha)$ and $\Stack_G(\beta)$ in direction $d$ we
write $\Stack_G(\alpha) <_{\mathcal{BL}_{G}^d} \Stack_G(\beta)$ iff
the first block of $\Stack_G(\alpha)$ is less (in the above defined order)
than the first block of $\Stack_G(\beta)$.
This ordering can be observed in the right part of Figure \ref{Ordering of blocks in a stack and ordering of stacks in a slice}.

\begin{figure}
	\centering
	\includegraphics[width=0.7\textwidth]{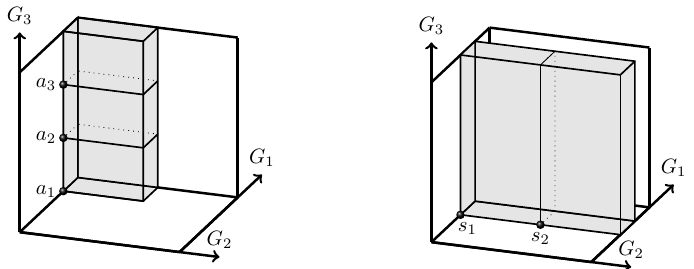}
	\caption{ Ordering of blocks in a stack and ordering of stacks in a slice. }
	\label{Ordering of blocks in a stack and ordering of stacks in a slice}
\end{figure}

Finally, we need to define an order on slices. One can view a slice as  
disjoint union of blocks. Alternatively, a slice can be viewed as 
disjoint union of stacks in the $d$-th direction 
(see Figure \ref{Ordering of blocks in a stack and ordering of stacks in a slice}).
We write $\Slice_G(i)<_{\mathcal{BL}_G^d}\Slice_G(j)$ iff $i<j$.

Summing up, $V_G$ is partitioned by slices, each slice is partitioned by 
stacks in the $d$-th direction, and each stack is partitioned by blocks.
Furthermore, there is a total order of slices, stacks, and blocks induced 
by the order $\mathcal{BL}_G^d$.

%auto-ignore
\section{Compression}
Let $G_1,\dots, G_d$ be graphs with some total orders 
$\mathcal{O}_{G_1},\dots, \mathcal{O}_{G_d}$ on their vertex sets, and let
$G=(V_G,E_G) = G_1\square \cdots \square G_d$.
For $S=\{i_1<\dots < i_k\}\subset \{1,\dots, d\}$ denote by $\mathcal{O}_{G_S}$ 
the induced order on $G_S = G_{i_1}\square \cdots \square G_{i_k}$.
Denote $\overline{S}=\{j_1<\dots <j_{d-k}\} = \{1,\dots, d\}\setminus S$.
For $x=(x_{j_1},\dots, x_{j_{d-k}})\in V_{G_{\overline{S}}}$
we define the \textit{cut} or \textit{section} of $G_S$ at $x$ to be the 
graph $G_{S}(x)=(V_{G_{S}(x)}, E_{G_{S}(x)})$,
where
\begin{align*}
    V_{G_{S}(x)} &= \{(v_1,\dots,v_d)\in V_G\bigm|\mbox{with } v_q=x_q \mbox
{ for } q\in \overline{S} \},\\
    E_{G_{S}(x)} &= I_G(V_{G_{S}(x)}).
\end{align*}
Note that the graphs $G_S(x)$ are isomorphic to $G_S$ for all $x$.
This isomorphism provides a total order $\mathcal{O}_{G_S(x)}$ on $G_S(x)$
induced by the order $\mathcal{O}_{G_S}$ on $G_S$.

For a set $A\subseteq V_G$ we define the \textit{compression} $ \Comp_{G,\mathcal{O}_{G_S}}(A)$ of $G$ with  respect to $\mathcal{O}_{G_S}$ as the operation that replaces the vertices in each $A \cap G_S$ with an initial segment (of order $\mathcal{O}_{G_S(x)}$ inside the cut $G_S(x)$)
of the same size. 
More formally we use the definition from \cite{H_Book},
\begin{align*}
    \Comp_{G,\mathcal{O}_{G_S}}(A) = \bigcup_{x\in V_{G_{ \{1,\dots, d\} 
    \setminus S}}} \mathcal{O}_{G_S(x)}[|A\cap V_{G_S(x)}|].
\end{align*}

The following lemmas \ref{downIsBetter} - \ref{weightCompressed} have been 
discovered and used by many authors, see, e.g. \cite{H_Book}, so we present
them without a proof here.

\begin{lemma}\label{downIsBetter}
If $\mathcal{O}_{G_S}$ is optimal then for any $A\subseteq V_G$ it holds:
\begin{enumerate}
    \item $|\Comp_{G,\mathcal{O}_{G_S}}(A)|=|A|$.
    \item If $B\subseteq A$ then 
    $\Comp_{G,\mathcal{O}_{G_S}}(B) \subseteq \Comp_{G,\mathcal{O}_{G_S}}(A)$.
    \item $|I_G(A)|\leq |I_G(\Comp_{G,\mathcal{O}_{G_S}}(A))|$.
\end{enumerate}
\end{lemma}

The next lemma informally says that if the orders in question are consistent, 
then after a finite time of applying the compression one gets a
\textit{stable} set.

\begin{lemma}\label{compProcess}
Let $S_0,\dots, S_{p-1} \subset \{1,\dots, d\}$ 
and $\mathcal{S}=(S_0,\dots , S_{p-1})$. For $A\subseteq V_G$ and $n\geq 1$ define
\begin{align*}
    \Comp_{G, \mathcal{S}}^n(A) = \begin{cases} 
      \Comp_{G, \mathcal{O}_{G_{S_0}}}(A) & \mbox{ if } n=1, \\
      \Comp_{G, \mathcal{O}_{G_{S_{n \bmod p}}}}(\Comp_{G, \mathcal{S}}^{n-1}(A) ) & \mbox{ if } n\geq 2. 
   \end{cases}
\end{align*}
If the order $\mathcal{O}_{G_{S_q}}$ is consistent with $\mathcal{O}_G$ for
$q=1,\dots,n$ then the sequence $(\Comp_{G, \mathcal{S}}^n(A))_{n=1}^\infty$ 
is eventually constant. In other words, there is $n_0$ such that for all
$n\geq n_0$ one has 
$\Comp_{G, \mathcal{S}}^{n+1}(A)=\Comp_{G,\mathcal{S}}^n(A)$.
\end{lemma}

Denote by $\Comp_{G,\mathcal{S}}(A)$ the resulting stable set in Lemma 
\ref{compProcess}. 
We say that $A\subseteq V_G$ is \textit{compressed} if for $\mathcal{S}=(\{1\},\dots, \{d\})$ we have
$\Comp_{G,\mathcal{S}}(A)=A$. 
Furthermore, we say that $A$ is \textit{strongly compressed} if 
$\Comp_{G,S}(A)=A$ for any proper subset $S$ of $\{1,\dots, d\}$.
In the sequel we will be looking for solutions to the edge-isoperimetric
problem that are compressed sets. 

For optimal orders $\mathcal{O}_{G_1},\dots,\mathcal{O}_{G_d}$ and $A\subseteq
V_G$ define the weight of $A$ as
\begin{align*}
    \omega_G(A) = \sum_{(i_1,\dots,i_n) \in A} \left( \sum_{j = 1}^n
\Delta_{G_j}(i_j)\right).
\end{align*}

\begin{lemma}\label{weightCompressed}
If $\mathcal{O}_{G_1},\dots, \mathcal{O}_{G_d}$ are optimal orders and $A\subseteq
V_G$ is a compressed set then
$|I_G(A)| = \omega_G(A)$.
\end{lemma}

%\begin{corollary}\label{posetEquiv}
%For $m\in\{1,\dots,|V_G|\}$ we have
%\begin{align*}
%    I_G(m)=
%    \max_{\substack{A\subseteq V_G \\ A \text{ compressed}\\ |A|=m}} \omega_G(A) =
%    \max_{\substack{A\subseteq V_G \\ A \text{ compressed}\\ |A|=m}} |I_G(A)|.
%\end{align*}
%\end{corollary}
%\begin{proof}
%Follows by the previous lemmas in this section.
%\end{proof}

This lemma immediately implies the following assertion.
\begin{corollary}\label{posetTrick}
Let $A\subseteq V_G$ be compressed, $T_1\subseteq V_G\setminus A$, and
$T_2\subseteq A$. If $(A\cup T_1)\setminus T_2$ is compressed then
\begin{align*}
    |I_G(A)| - |I_G((A\cup T_1)\setminus T_2)| = \omega_G(A)-\omega_G((A\cup T_1)\setminus T_2)
    = \omega_G(T_2) - \omega_G(T_1).
\end{align*}
\end{corollary}

We extend compression to the geometric objects introduced in the previous
section. Suppose that $A\subseteq V_G$ is compressed and there is a 
domination collection $\mathfrak{P}_{G_1},\dots, \mathfrak{P}_{G_d}$. 
Let $B_1,\dots, B_p$ be all blocks of $G$ ordered so that
$B_a <_{\mathcal{BL}_G^d} B_b$ whenever $a<b$.
Denote by $r$ be the largest block number such that $A\cap B_r\neq\emptyset$.
We say that $A$ is \textit{block compressed} if $B_i\subseteq A$ for $i<r$.

Consider a slice $\Slice_G(q)$ and let $B_1,\dots, B_{p_q}$ be its blocks 
ordered so that $B_a <_{\mathcal{BL}_G^d} B_b$ whenever $a<b$. 
Denote by $r_q$ be the largest block number such that 
$A\cap B_{r_q}\neq \emptyset$.
We say that $A$ is \textit{slice compressed} if for each
$q=1,\dots,|\mathfrak{P}_{G_1}|$ and $i<r_q$ it holds 
$B_i\subseteq A$. Thus, if $A$ is slice compressed then it is block
compressed.

%auto-ignore
\section{Proof of the Main result}
Let $G_1,\dots,G_d$ be isoperimetric graphs with a regular domination
collection of non-decreasing isoperimetric partitions 
$\mathfrak{P}_{G_1},\dots,\mathfrak{P}_{G_d}$. We also assume that
the order $\mathcal{BL}_{G_i\square G_j}^{2}$ is optimal for all
$i,j\in \{1,\dots, d\}$ with $i<j$ and let $A\subseteq V_G$ be an optimal
set.

The theorem is proved by introducing a series of operations that transform the
set $A$ into the initial segment of order $\mathcal{BL}_G^d$ of the same size 
without reducing the number of induced edges. First, we make $A$ 
slice-compressed (Theorem \ref{sliceCompression}), then block-compressed 
(Theorem \ref{last_of_final_proof}), and finally use a special transformation 
of the resulting set.
We assume that the theorem holds for all $d'<d$ and proceed by induction on
$d$ for $d\geq 3$.

\begin{lemma}\label{Consistency}
The order $\mathcal{BL}_G^d$ is consistent with $\mathcal{BL}_{G_S}^k$ 
for any $S=\{i_1<\dots < i_k\}\subset\{1,\dots,d\}$.
\end{lemma}
\begin{proof}
Let $x=(x_1,\dots,x_d)$ and $y=(y_1,\dots, y_d)$ be vertices of $V_G$ with 
$x <_{\mathcal{BL}_G^d} y$ and $x_j=y_j$ for $j\not\in S$.
Denote $x'=(x_{i_1},\dots,x_{i_k})$ and $y'=(y_{i_1},\dots,y_{i_k})$.
If $\Block_{G}(x) <_{\mathcal{BL}_{G_S}^d} \Block_{G}(y)$ then 
$\Block_{G_S}(x') <_{\mathcal{BL}_{G_S}^k} \Block_{G_S}(y')$
since lexicographic order is consistent.
This implies $x' <_{\mathcal{BL}_{G_S}^k} y'$.

If $\Block_{G}(x) = \Block_{G}(y)$ then $\Block_{G_S}(x') = \Block_{G_S}(y')$,
and the domination order $\mathcal{D}_{\Block_{G}(x)}$ is consistent with
the domination order $\mathcal{D}_{\Block_{G_S}(x')}$,
since the partitions $\mathfrak{P}_{G_1},\dots, \mathfrak{P}_{G_d}$ 
form a domination collection. Thus, $x'<_{\mathcal{BL}_{G_S}^k} y'$.
\end{proof}

This lemma along with lemmas \ref{downIsBetter}, \ref{compProcess} implies
that there is no loss of generality to assume that $A\subseteq V_G$ is
strongly compressed. To make $A$ slice-compressed, several auxiliary results
are needed. We say that two blocks $B_1$ and $B_2$ in the same stack
$\Stack_G(\alpha)$ are \textit{consecutive},
if there is no block $B_3\subseteq \Stack_G(\alpha)$
such that $B_1 <_{\mathcal{BL}_G^d} B_3 <_{\mathcal{BL}_G^d} B_2$.

\begin{lemma}\label{nonEmptyImpliesSkeleton}
Let $A\subseteq V_G$ be strongly compressed and $B_1 <_{\mathcal{BL}_G^d} B_2$ be
consecutive blocks of a stack in some direction $i$.
If $B_2\cap A \neq \emptyset$ then $\Skeleton_G(B_1)\subseteq A$. 
\end{lemma}
\begin{proof}
Note that $\Start_G(B_2)\in A$, since $B_2\cap A \neq \emptyset$ 
and $A$ is strongly compressed.
We show that all bones of $B_1$ are in $A$.
%First, note that $\Bone_G(B_1,i)\subseteq A$,
%since $\Start_G(B_2)\in A$ and $A$ is compressed.
Let $\Start_G(B_1)=(s_1,\dots, s_d)$ and 
$x=(s_1,\dots,s_{j-1},x_j,s_{j+1},\dots, s_d)\in\Bone_G(B_1,j)$ for 
some $j\in \{1,\dots, d\}$. 
So, at least $d-1$ coordinates of $x$ and $\Start_G(B_1)$ are the same.
Also, all coordinates except the $i$-th one of $\Start_G(B_1)$ and 
$\Start_G(B_2)$ match because the blocks are in the same stack in 
direction $i$.
Hence, $\Start_G(B_2)$ and $x$ share at least $d-2\geq 1$ equal coordinates.
Since $A$ is strongly compressed, $x <_{\mathcal{BL}_G^d} \Start_G(B_2)$,
and $\Start_G(B_2)$ and $x$ match in at least $1$ coordinate, 
we conclude $x\in A$.
\end{proof}

\begin{lemma}\label{sharedSegments}
Let $A\subseteq V_G$ be strongly compressed and blocks $B_1$ and $B_2$ with
$B_1 <_{\mathcal{BL}_G^d} B_2$ share the $i$-th bone for some 
$i\in \{1,\dots, d\}$ and $\Bone_G(B_2,i)\subseteq A$. Then $B_1\subseteq A$.
\end{lemma}
\begin{proof}
Let $\Start_G(B_2)=(s_1,\dots, s_d)$ and $x=(x_1,\dots,x_i,\dots,x_d)\in B_1$.
Then we have
\begin{align*}
	y=(s_1,\dots, s_{i-1}, x_i, s_{i+1},\dots, s_d)\in \Bone_G(B_2,i).
\end{align*}
Since $x$ and $y$ match in the $i$-th coordinate and $x<_{\mathcal{BL}_G^d} y$ we
conclude $x\in A$. This implies $B_1\subseteq A$.
\end{proof}

\begin{corollary}\label{skelImpliesEmpty}
Let $A\subseteq V_G$ be strongly compressed and blocks 
$B_1 <_{\mathcal{BL}_G^d} B_2$ share a bone. Let $\Stack_{G}(\alpha)$ be a stack
containing $B_2$ and $B_2 <_{\mathcal{BL}_G^d} B_3$ for some block
$B_3\subseteq \Stack_G(\alpha)$.
If $B_1 \not \subseteq A$ and $B_2\cap A \neq \emptyset$ then 
$B_3\cap A =\emptyset$.
\end{corollary}
\begin{proof}
For the contrary, assume $B_3\cap A \neq \emptyset$.
Then $\Skeleton_G(B_2)\subseteq A$, by Lemma \ref{nonEmptyImpliesSkeleton}.
Since $B_1$ and $B_2$ share a bone, Lemma \ref{sharedSegments} implies 
$B_1\subseteq A$, which is a contradiction.
\end{proof}

For a slice $\Slice_G(q)$ and blocks $B_1 <_{\mathcal{BL}_G^d} B_2$ in it we say
$B_1$ and $B_2$ are \textit{consecutive} in slice $\Slice_G(q)$
if there is no block $B_3$ in $\Slice_G(q)$ with 
$B_1 <_{\mathcal{BL}_G^d} B_3 <_{\mathcal{BL}_G^d} B_2$.

\begin{lemma}\label{cubesInSliceConsecutive}
Let blocks $B_1 <_{\mathcal{BL}_G^d} B_2$ be in slice $\Slice_G(q)$.
If $B_1\not\subseteq A$ and $A\cap B_2 \neq \emptyset$,
then $B_1$ and $B_2$ are consecutive in $\Slice_G(q)$.
\end{lemma}
\begin{proof}
Let $\Stack_G(\alpha)$ and $\Stack_G(\beta)$ be stacks of $\Slice_G(q)$ in
direction $d$ that contain $B_1$ and $B_2$, respectively.
If $\Stack_G(\alpha) = \Stack_G(\beta)$ then the statement follows from
lemmas \ref{nonEmptyImpliesSkeleton} and \ref{sharedSegments}.
So assume $\Stack_G(\alpha) \neq \Stack_G(\beta)$.

Let $B_1'$ be the last block of $\Stack_G(\alpha)$ and $B_2'$ be the first 
block of $\Stack_G(\beta)$. We show that $B_1'=B_1$ and $B_2'=B_2$.
Indeed, if $B_2'\neq B_2$ then since $\Skeleton_G(B_2')\subseteq A$ by 
Lemma \ref{nonEmptyImpliesSkeleton}, we get $B_1\subseteq A$ by 
Lemma \ref{sharedSegments}, a contradiction.
Also, if $B_1'\neq B_1$ then by the definition of slices, $B_2$ and $B_1'$ 
share the first bone. 
Since $B_2\cap A \neq \emptyset$ and $A$ is strongly compressed,
for all $i\in \{2,\dots, d\}$ we have $\Bone_G(B_1',i)\subseteq A$.
Thus, $B_1\subseteq A$ by Lemma \ref{sharedSegments},
since $d\geq 3$ and $B_1$ and $B_1'$ share the $j$-th bone for all 
$j\in \{1,\dots, d-1\}$. This implies $B_1\subseteq A$, a contradiction.

It remains to show is that there is no stack $\Stack_G(\gamma)$
with $\Stack_G(\alpha)<_{\mathcal{BL}_G^d}\Stack_G(\gamma)<_{\mathcal{BL}_G^d}
\Stack_G(\beta)$. Assume for the contrary that this is not the case and let
$B_3$ be the last block in $\Stack_G(\gamma)$.
We get $\Bone_G(B_3,j)\subseteq A$ for all $j\in \{2,\dots, d\}$,
since $A$ is strongly compressed and $B_3$ and $B_2$ share the first bone.
If $B_3$ is the only block in $\Stack_G(\gamma)$, then $B_1$ and $B_3$ 
share the $d$-th bone. Hence, $B_1\subseteq A$ by Lemma \ref{sharedSegments},
a contradiction.
So, suppose that there is another block $B_4\subseteq \Stack_G(\gamma)$.
Then $B_3$ and $B_4$ share the $j$-th bone for all $j\in \{1,\dots, d-1\}$.
Lemma \ref{sharedSegments} implies $B_4\subseteq A$, since $d\geq 3$.
However, Lemma \ref{sharedSegments} implies $B_1\subseteq A$,
since $B_1$ and $B_4$ share the first bone. The obtained contradiction
completes the proof.
\end{proof}

The above lemmas are used as a basis for establishing the next results by using
the pull-push method.

\begin{theorem}\label{sliceCompression}
For any strongly compressed set $A\subseteq V_G$ there exist 
a strongly compressed and slice-compressed set $B\subseteq V_G$ such that
$|A|=|B|$ and $|I_G(A)|\leq |I_G(B)|$.
\end{theorem}
\begin{remark}
The proof of Theorem \ref{sliceCompression} has three steps.
First, we use Lemma \ref{cubesInSliceConsecutive} to consider two consecutive cubes.
Then we are going to move vertices to the earlier block from the later one.
This is done in two steps, we first do a pull and then a push.
The pull introduces a new set $A'$ that is compressed similar to $A$, but not strongly compressed.
We then use compression on $A'$ to pull vertices to the earlier block from the later one and obtain a new set $D'$.
The push deals with how we transfer information gained from the pull on $A'$ to the set $A$.
The push compares the pulled vertices with corresponding vertices that come later in the block lexicographic order.

Properties of strong compression imply that the number of induced edges by a set obtained after pulling cannot decrease. 
Similarly, the non-decreasing property of isoperimetric partitions of graphs in the product guarantee that the pushing operation also does not decrease the number of induced edges. 
This way, applying both operations to an optimal set results in another optimal set satisfying some structural properties which make it looking closer to an initial segment of the block lexicographic order.
\end{remark}

We use the pull-push method three more times in Theorem \ref{last_of_final_proof}. 

\begin{proof}
Let $B_1<_{\mathcal{BL}_G^d}\dots<_{\mathcal{BL}_G^d} B_{p}$ be the blocks of some
slice $\Slice_G(q)$ and let $r$ be the largest index for which 
$A\cap B_{r} \neq \emptyset$. 
Lemma \ref{cubesInSliceConsecutive} implies $B_1,\dots, B_{r-2}\subseteq A$.
Omitting trivial cases we assume $r\geq 2$ and $B_{r-1}\not\subseteq A$.
We will pump vertices from $B_{r}$ to $B_{r-1}$ by using the pull-push method.
	
Let $S=\{2,\dots, d\}$ and $B_{r}=Z_1\times \cdots \times Z_d$. Note 
that $Z_1$ is the partition of $G_1$ that all blocks in $\Slice_G(q)$ share.
Denote by $x\in Z_1$ the first vertex in the order $\mathcal{O}_{G_1}$,
such that $V_{G_S(x)}\cap B_{r-1}\not\subseteq A$ and by $y\in Z_1$ the last
vertex in the order $\mathcal{O}_{G_1}$ such that $V_{G_S(y)}\cap 
B_{r}\cap A\neq \emptyset$. 
Therefore, for all $v\in Z_1$ with $v<_{\mathcal{O}_{G_1}} x$ we have 
$V_{G_S(v)}\cap B_{r-1}\subseteq A$ and for all $v\in Z_1$ with 
$v>_{\mathcal{O}_{G_1}} y$ we have $V_{G_S(v)}\cap B_{r}\cap A = \emptyset$.
Note that $x$ is defined with respect to $B_{r-1}$ and $y$ with respect 
to $B_{r}$. Also note that $x >_{\mathcal{O}_{G_1}} y$, since otherwise
$V_{G_S(x)}\cap B_{r-1}\subseteq A$ because $A$ is strongly compressed.
We are going to pump vertices from $V_{G_S(y)}\cap B_{r}$ to 
$G_S(x)\cap B_{r-1}$ in the two following steps.
	
{\it Pull:}
Let $W$ be the projection of $V_{G_S(x)} \cap B_{r-1}\cap A$
to $V_{G_S(y)} \cap B_{r-1}$, that is
\begin{align*}
	W &= \{(y, v_2,\dots, v_d) \bigm | 
	(x,v_2\dots, v_d)\in  G_S(x) \cap B_{r-1}\cap A  \}.
\end{align*}
Denote by $E$ the set of all $d-1$ dimensional blocks in $B_{r-1}$ such that
\begin{align*}
	E &= \bigcup_{\substack{v\in Z_1 \\ v \geq_{\mathcal{O}_{G_1}} y}} 
	V_{G_S(v)}\cap B_{r-1},
\end{align*}
and denote by $R$ the set of all slices greater than $\Slice_G(q)$, that is
\begin{align*}
	R &= \bigcup_{t>q} \Slice_G(t).
\end{align*}
Consider the set
\begin{align*}
	A' = (A\setminus(E\cup R)) \cup W.
\end{align*}
This set has the following properties:
\begin{enumerate}
\item $A'$ is compressed. 
To prove this we show that if $v=(v_1,\dots, v_d)\in A'$ and 
$u=(u_1,\dots, u_d)\in V_G$ with 
$u_1\leq_{\mathcal{O}_{G_1}} v_1,\dots, u_d\leq_{\mathcal{O}_{G_d}} v_d$,
then $u\in A'$.
First, note that $B_{r}$ and $B_{r-1}$
are the last blocks that have a nonempty intersection with $A'$,
since we removed $R$ from $A$ to get $A'$.
If $v\not \in B_{r}$ or $v\not \in B_{r-1}$ then $u\in A'$,
since $A$ is strongly compressed and we only modified $B_{r-1}$ to get $A'$.
If $v\in B_{r-1}$ then $u\in A'$, since $W\subseteq G_S(y)\cap B_{r-1}\cap A$.
So, it remains to consider the case $v\in B_{r}$ and $u\in B_{r-1}$.
Actually, it is sufficient to consider only the case when 
$B_{r}$ and $B_{r-1}$ are in the same stack,
since otherwise there is some $i\in \{1,\dots, d\}$
for which $u_i >_{\mathcal{O}_{G_i}} v_i$.
So, suppose that $B_{r}$ and $B_{r-1}$ are in the same stack in direction $i$.
Note that $i\neq 1$ because $B_{r}$ and $B_{r-1}$ are in the same slice.
Without loss of generality, assume that
$u_1=v_1,\ \dots,\  u_{i-1}=v_{i-1},\ u_{i+1}=v_{i+1},\ \dots,\ u_d=v_d$.
Since $A$ is strongly compressed, we get $u\in A$.
If $u_1<_{\mathcal{O}_{G_1}} y$ then $u\in A'$, by the definition of $A'$ 
because $u\not \in E$.
If $u_1=y$ then $(x,u_2,\dots, u_d)\in A$ because $A$ is strongly 
compressed and $d\geq 3$. Therefore, in all cases $u\in A'$ by the 
definition of $W$ and $A'$, implying $A'$ is compressed.
\item The set $A'\cap V_{G_S(y)}\cap B_{r-1}$ forms an initial segment of 
the domination order in block $V_{G_S(y)}\cap B_{r-1}$.
\item The set $A'\cap V_{G_S(y)}\cap B_{r}$ forms an initial segment of 
the domination order in block $V_{G_S(y)}\cap B_{r}$.
\item For all $l<r-1$ it holds that $B_l\subseteq A'$.
\item For all $l>r$ it holds that $B_l\cap A' =\emptyset$.
\end{enumerate}

Denote $D'=\Comp_{G,\mathcal{O}_{G_S}}(A')$. The set $D'$ is obtained from $A'$
by moving $n$ vertices from $B_{r}$ to $B_{r-1}$, where
\begin{align*}
	a = \min\{|A'\cap (V_{G_S(y)}\cap B_{r})|, 
	|(V_{G_S(y)}\cap B_{r-1})\setminus (A'\cap G_S(y)\cap B_{r-1})| \}
\end{align*}
Note that $D'$ is compressed, since $A'$ is compressed. Denote by $T_r$ the
set of the last $a$ vertices of $A'\cap V_{G_S(y)}\cap B_{r}$ in the 
domination order on $V_{G_S(y)}\cap B_{r}$, and denote by $T_{r-1}$ the
set of the first $a$ vertices of $(V_{G_S(y)}\cap B_{r-1})\setminus A'$ in the domination 
order on $V_{G_S(y)}\cap B_{r-1}$. In these terms, $D' = (A'\setminus
T_r)\cup T_{r-1}$.
Taking into account Corollary \ref{posetTrick} one has
\begin{align*}
	0 \leq |I_G(D')|-|I_G(A')| = \omega_G(T_{r-1})-\omega_G(T_{r}).
\end{align*}

{\it Push:}
Denote 
\begin{align*}
	T_{x} &= \{(x,v_2,\dots, v_d) \bigm | (y,v_2,\dots, v_d)\in T_{r-1}\}\\
	D &= (A\setminus T_r)\cup T_{x}.
\end{align*}
Note that block $B_{r}$, as well as any other block, belongs to $d$ stacks 
of $G$ (in different directions).
If there is a block $B$ in one of those stacks such that 
$B_{r}<_{\mathcal{BL}_G^d} B$ then $B\cap A = \emptyset$ by 
Corollary \ref{skelImpliesEmpty} because $B_{r-1}$ and $B_r$ share a bone.
Therefore, $A\setminus T_r$ is compressed. By definition of $x$ we get that 
$D$ is also compressed.
	
Now, Corollary \ref{posetTrick}, Lemma \ref{deltaInPartition}, and 
the non-decreasing property of $\mathfrak{P}_{G_1}$ imply
\begin{align*}
  |I_G(D)|-|I_G(A)| &= \omega_G(T_x)-\omega_G(T_{r})\\
  &=\omega_G(T_x) -\omega_G(T_{r-1}) + \omega_G(T_{r-1})-\omega_G(T_{r})\\
  &\geq a(\Delta_{G_1}(x)-\Delta_{G_1}(y)) +0\\
  &=a(\Delta_{(Z_1, I_{G_1}(Z_1))}(x) - \Delta_{(Z_1, I_{G_1}(Z_1))}(y))\\
  &\geq 0.
\end{align*}
Apply to $D$ the strong compression operation and denote by $F$ the resulting
set. One has $|F|=|D|$ and $|I_G(F)|\geq |I_G(D)|$.
Applying the described transformation over and over we obtain a stable set
$B$ of the same size, which is strongly compressed and slice compressed, 
and for which $|I_G(A)|\leq |I_G(B)|$.
\end{proof}

We now apply a similar approach for reducing the problem of constructing 
optimal sets to block compressed sets.

\begin{lemma}\label{lineImpliesSlice}
Let $A\subseteq V_G$ be a strongly compressed set and let $(s_1,\dots, s_d)$ 
be the start of the first block of some slice $\Slice_G(q)$.
If for $i\in \{2,\dots, d\}$ it holds
\begin{align*}
	\{s_1\}\times \cdots \times \{s_{i-1}\}\times V_{G_i}\times \{s_{i+1}\}
	\times \cdots \times \{s_d\} \subseteq A,
\end{align*}
then $\Slice_G(p)\subseteq A$ for all $p< q$.
\end{lemma}
\begin{proof}
Indeed, if $x=(x_1,\dots,x_i,\dots, x_d)\in \Slice_G(p)$ then
$(s_1,\dots, s_{i-1},x_i,s_{i+1},\dots, s_d)\in A$ and 
$x_1 <_{\mathcal{O}_{G_1}} s_1$. 
Hence, $x\in A$, since $A$ is strongly compressed.
\end{proof}

\begin{lemma}\label{consecutiveSlices}
Let $A\subseteq V_G$ be strongly compressed and slice-compressed and
slices $\Slice_G (p) <_{\mathcal{BL}^d_G} \Slice_G (q)$ be such that 
$\Slice_G(q)\cap A \neq \emptyset$ and $\Slice_G(p)\not\subseteq A$.
If $\Stack_G(\alpha_q)$ and $\Stack_G(\alpha_p)$ are the first and last 
stacks of $\Slice_G(q)$ and $\Slice_G(p)$, respectively, then
\begin{enumerate}
\item For any stack $\Stack_G(\beta) \subseteq \Slice_G(q)$ 
different from $\Stack_G(\alpha_q)$ it holds
$\Stack_G(\beta)\cap A = \emptyset$.
\item For any stack $\Stack_G(\beta) \subseteq \Slice_G(p)$ 
different from $\Stack_G(\alpha_p)$, it holds 
$\Stack_G(\beta)\subseteq A$.
\item The slices $\Slice_G (p)$ and $\Slice_G (q)$ are consecutive,
that is, $q=p+1$.
\end{enumerate}
\end{lemma}

\begin{proof}
Assume to the contrary that the first claim does not hold.
Since $A$ is slice compressed, one has $\Stack_G(\alpha_q)\subseteq A$.
By Lemma \ref{lineImpliesSlice}, $\Slice_G(p)\subseteq A$, 
a contradiction.
For the second one note that $\Slice_G(q) \cap A \neq \emptyset$ implies 
$\Stack_G(\alpha_p)\cap A \neq \emptyset$, since $A$ is strongly compressed.
Thus, all stacks preceding $\Stack_G(\alpha_p)$ must be in $A$.
For the last statement, assume that $q>p+1$.
Then $\Slice_G(p) <_{\mathcal{BL}_G^d} \Slice_G(q-1)$.
Let $(s_1,\dots, s_d)$ be the start of the first block in $\Slice_G(q-1)$.
Since $A$ is strongly compressed we have
\begin{align*}
	\{s_1\}\times \cdots \times \{s_{d-1}\}\times V_{G_d} \subseteq A.
\end{align*}
For $x=(x_1,\dots, x_d)\in \Slice_G(p)$ one has $(s_1,\dots,s_{d-1},x_d)\in A$
and $x_1<s_1$. Hence, $x\in A$, since $A$ is strongly compressed.
Therefore, $\Slice_G(p)\subseteq A$, a contradiction.
\end{proof}

\begin{theorem}\label{last_of_final_proof}
For any strongly and slice compressed set $A\subseteq V_G$ 
there exists a strongly and block compressed set $B\subseteq V_G$ such that
$|A|=|B|$ and $|I_G(A)|\leq |I_G(B)|$.
\end{theorem}
\begin{proof}
Let $q$ be the largest integer such that $\Slice_G(q)\cap A\neq\emptyset$.
Omitting trivial cases we assume $q>1$ and $\Slice_G(q-1)\not\subseteq A$.
By Lemma \ref{consecutiveSlices}, if $q\geq 3$ we have 
$\Slice_G(1),\dots, \Slice_G(q-2)\subseteq A$.
Let $\Stack_G(\alpha_q)$ be the first stack of $\Slice_G(q)$,
and $\Stack_G(\alpha_{q-1})$ be the last stack in $\Slice_G(q-1)$.
Note that we always take stacks in the $d$-th direction when talking 
about them in slices. Lemma \ref{consecutiveSlices} tells that for 
every stack $\Stack_G(\beta)\subseteq \Slice_G(q)$ different from
$\Stack_G(\alpha_q)$ one has $\Stack_G(\beta)\cap A = \emptyset$. 
By the same lemma, $\Stack_G(\beta)\subseteq A$ for every stack
$\Stack_G(\beta)\subseteq \Slice_G(q-1)$ different from 
$\Stack_G(\alpha_{q-1})$.
We will apply pull-push approach to pump vertices from 
$\Stack_G(\alpha_q)$ to $\Stack_G(\alpha_{q-1})$.
	
Let $B_q = Z_1\times \cdots \times Z_d$ and $B_{q-1} = Y_1\times \cdots
\times Y_d$  be the first blocks of $\Stack_G(\alpha_q)$ and 
$\Stack_G(\alpha_{q-1})$, respectively. Denote
\begin{align*}
 H_Z&=(Z_2,I_{G_2}(Z_2))\square \cdots \square (Z_{d-1},I_{G_{d-1}}(Z_{d-1})),\\
 H_Y&=(Y_2,I_{G_2}(Y_2))\square \cdots \square (Y_{d-1},I_{G_{d-1}}(Y_{d-1})).
\end{align*}
	
\begin{figure}
\centering
\includegraphics[width=0.5\textwidth]{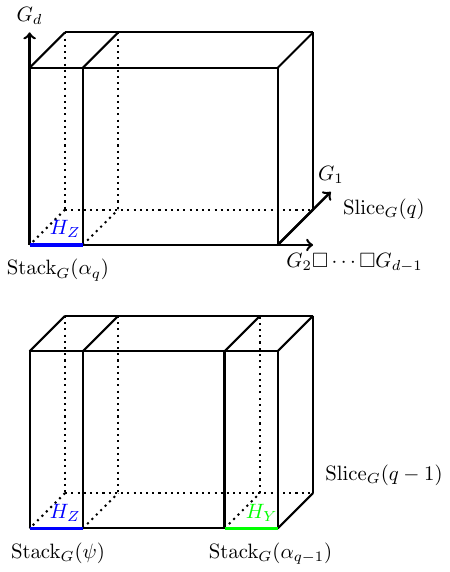}
\caption{Two consecutive slices and stacks that appear in the proof 
of Theorem \ref{last_of_final_proof}.}
\label{theorem 5.2 set ups}
\end{figure}
	
It follows that $Z_i$ and $Y_i$ are, respectively, the first and the last 
parts of the partition $\mathfrak{P}_{G_i}$ for $i\in\{2,\dots, d-1\}$.
These notations are illustrated in Figure \ref{theorem 5.2 set ups}.
Remember that according to our assumption, the partitions 
$\mathfrak{P}_{G_1},\dots,\mathfrak{P}_{G_d}$ form a regular domination
collection, which means that:

\begin{enumerate}
\item For $i=2,\dots,d-1$ the partition $\mathfrak{P}_{G_i}$ is 
regular, i.e., $\delta_{(Z_i,I_{G_i}(Z_i))} = \delta_{(Y_i,I_{G_i}(Y_i))}$.
\item There is a $\tau \in \mathfrak{S}_{d-2}$ such that the domination orders 
on $H_Z$ and $H_Y$ are both induced by $\mathcal{D}^{\tau, d-2}$.
Denote $\mathcal{D}=\mathcal{D}^{\tau, d-2}$ for brevity.
\end{enumerate}
%Moreover, we assumed that for $i=1,\dots,d-1$ the partition 
%$\mathfrak{P}_{G_i}$ is non-decreasing.
	
Set $S=\{1,d\}$ and let $y\in V_{H_Y}$ be the first vertex in the order 
$\mathcal{D}_{H_Y}$ for which $V_{G_S(y)}\cap \Stack_G(\alpha_{q-1})\not\subseteq A$.
Also, let $z\in V_{H_Z}$ be the last vertex in the order $\mathcal{D}_{H_Z}$
such that $V_{G_S(z)}\cap \Stack_G(\alpha_q)\cap A \neq \emptyset$. It
follows that $V_{G_S(v)}\cap \Stack_G(\alpha_{q-1}) \subseteq A$
for every $v\in V_{H_Y}$ with $v<_{\mathcal{D}_{H_Y}} y$ and $V_{G_S(v)}\cap
\Stack_G(\alpha_q)\cap A = \emptyset$ for every $v\in V_{H_Z}$ with
$v>_{\mathcal{D}_{H_Z}} z$. Denote $k=\mathcal{D}_{H_Y}(y)$ and $l=\mathcal{D}_{H_Z}(z)$ and
let $\Stack_G(\psi)$ be the first stack of the slice $\Slice_G(q-1)$.
	
{\it Case 1:} Assume $k < l$.
Let $W$ be the projection of $A\cap \Stack_G(\alpha_{q-1})$ to 
$\Stack_G(\psi)$, i.e.,
\begin{align*}
  W = \{(v_1, \mathcal{D}_{H_Z}^{-1}(\mathcal{D}_{H_Y}(u)), v_d) \bigm | 
        (v_1,u,v_d)\in \Stack_G(\alpha_{q-1})\cap A\}.
\end{align*}
Consider the set $A'$,
\begin{align*}
	A' = (A\setminus \Slice_G(q-1))\cup W.
\end{align*}
Set $A'$ has the following:
\begin{enumerate}
\item $A'$ is compressed. 
To show this, let $v = (v_1,v_2,\dots,v_d) \in A'$ and 
$u = (u_1,u_2,\dots,u_d) \in V_G$ such that $\exists i \in \{1,2,\dots,d\}$ 
with $u_i <_{\mathcal{O}_{G_i}} v_i$ and $u_j = v_j$ for all $j \neq i$.
Without loss of generality we assume that $u \not\in \Slice_G(p)$ for 
$p < q - 1$ since $\Slice_G(p) \subseteq A'$.
First, suppose $v \in \Slice_G(q-1)$. If $u \in \Slice_G(q-1)$, 
then $u,v\in\Stack_G(\psi)$, hence $u \in A'$ as $A$ is compressed.
Now suppose $v \in \Slice_G(q)$. If $u \in \Slice_G(q)$, then 
$u\in A'$ since $A$ is compressed. If $u \in \Slice_G(q-1)$ then $i=1$,
hence $u = (u_1,v_2,v_3,\dots,v_d)$. Since $\Slice_G(q) \cap A' = 
\Slice_G(q) \cap A$, we have $v \in A$.
So, any vertex $(u_1',u_2',\dots,u_{d-1}',v_d)\in V_G\cap\Slice_G(q-1)$ 
with $u_1' <_{\mathcal{O}_{G_1}} v_1$ is in $A$ as $A$ is strongly compressed.
In particular, $(u_1, \mathcal{D}_{H_Y}^{-1}(\mathcal{D}_{H_Z}(u_2,u_3,\dots,u_{d-1})), v_d) 
\in A$. This implies $u \in A'$.

\item For every $v\in V_{H_Z}$ the set $V_{G_S(v)}\cap A'\cap \Stack_G(\alpha_q)$ 
is an initial segment of order $\mathcal{BL}_{G_S(v)}^2$ restricted to 
$\Stack_G(\alpha_q)$, since $A$ is strongly compressed.

\item For every $v\in V_{H_Z}$ the set $V_{G_S(v)}\cap A'\cap \Stack_G(\psi)$ 
is an initial segment of order $\mathcal{BL}_{G_S(v)}^2$ restricted to 
$\Stack_G(\psi)$, since $A$ is strongly compressed and by the definition of $W$.
\end{enumerate}

Construct $D' = \Comp_{G,S}(A')$ (see Figure \ref{pulling method}) and denote
by $y'\in V_{H_Z}$ the vertex corresponding to $y$, i.e., $y'=\mathcal{D}_{H_Z}^{-1}(k)$. 
The set $D'$ can be constructed by moving vertices from 
$V_{G_S(v)}\cap A'\cap \Stack_G(\alpha_q)$ to 
$(V_{G_S(v)}\cap \Stack_G(\psi))\setminus A'$, 
for all $v\in R_{H_Z}$, where 
$R_{H_Z}=\{v\in V_{H_Z}\bigm| y' \leq_{\mathcal{D}_{H_Z}} v \leq_{\mathcal{D}_{H_Z}} z\}$.
For $v\in R_{H_Z}$ denote by $a_v$ the number of vertices moved,
\begin{align*}
	a_v = \min \{|V_{G_S(v)}\cap A'\cap \Stack_G(\alpha_q)|,
	|(V_{G_S(v)}\cap \Stack_G(\psi))\setminus A'|\}.
\end{align*}
	
\begin{figure}
\centering
\includegraphics[width=0.6\textwidth]{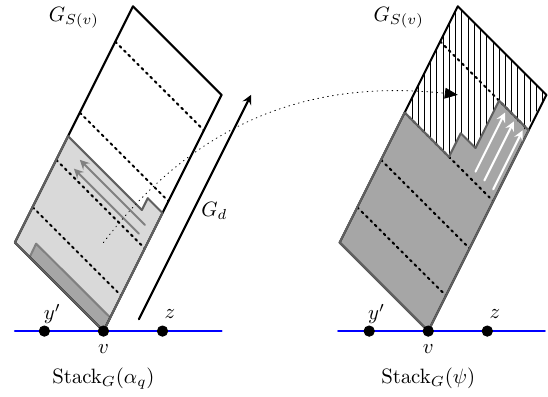}
\caption{Construction of $D'$ in case 1 of the proof.}
\label{pulling method}
\end{figure}

The set $D'$ can be constructed from $A'$ by applying the following two steps
for all $v\in R$: 
\begin{enumerate}
\item Remove the set $T_{q,v}$ consisting of the last $a_v$ 
vertices of $V_{G_S(v)}\cap A'\cap \Stack_G(\alpha_q)$ in the order 
$\mathcal{BL}_{G_S(v)}^2$ restricted to $\Stack_G(\alpha_q)$. 
\item Add the set $T_{q-1,v}$ consisting of the first $a_v$ vertices in 
$(V_{G_S(v)} \cap \Stack_G(\psi))\setminus A'$ 
in the order $\mathcal{BL}_{G_S(v)}^2$ restricted to $\Stack_G(\psi)$. 
\end{enumerate}
Corollary \ref{posetTrick} implies
\begin{align*}
	0 \leq |I_G(D')|-|I_G(A')| = 
	\sum_{v\in R_{H_Z}} \omega_G(T_{q-1,v})-\omega_G(T_{q,v}).
\end{align*}
For $v\in R_{H_Z}$ denote
\begin{align*}
 T_{v} = \{ (v_1, \mathcal{D}_{H_Y}^{-1}(\mathcal{D}_{H_Z}(u)), v_d)\bigm | (v_1,u, v_d)\in T_{q-1,v}\}
\end{align*}
and construct the set	
\begin{align*}
	D = \left( A \setminus \bigcup_{v\in R_{H_Z}} T_{q,v}\right)\cup
      	\bigcup_{v\in R_{H_Z}} T_v.
\end{align*}
We show that $D$ is compressed. Indeed, let $v = (v_1,v_2,\dots,v_d) \in D$ 
and $u = (u_1,u_2,\dots,u_d) \in V_G$ such that 
$\exists  i \in \{1,2,\dots,d\}$ with $u_i <_{\mathcal{O}_{G_i}} v_i$ and 
$u_j = v_j$ for all $j \neq i$.
Without loss of generality we assume $u \not\in \Slice_G(p)$ for all 
$p < q-1$ since $\Slice_G(p) \subseteq D$.
If $v \in \Slice_G(q - 1)$, then $u \in \Slice_G(q-1)$ by the previous 
assumption.

If $u \not\in \Stack_G(\alpha_{q-1})$, then $u \in D$ since for any stack 
$\Stack_G(\beta) \subseteq \Slice_G(q-1)$ different from 
$\Stack_G(\alpha_{q-1})$ we have $\Stack_G(\beta) \subseteq D$.
So suppose $u \in \Stack_G(\alpha_{q-1})$. Then $u \in D$ since 
$\Stack_G(\alpha_{q-1}) \cap D$ is a projection of $\Stack_G(\psi)\cap D'$
which is in the compressed set $D'$.
If $v \in \Slice_G(q)$ and $u \in \Slice_G(q)$, then $u \in D$ as 
$D \cap \Slice_G(q) = D' \cap \Slice_G(q)$ and $D'$ is compressed.
Finally, if $v \in \Slice_G(q)$ and $u \in \Slice_G(q-1)$ we may assume
$u \in \Stack_G(\psi)$.
Then $i = 1$, hence $u = (u_1,v_2,v_3,\dots,v_d)$.
Since $u \in \Stack_G(\psi) \subseteq D$, we get $u \in D$. 
Therefore, $D$ is compressed.

One has,
\begin{align*}
|I_G(D)| - |I_G(A)| &= \sum_{v\in R_{H_Z}} \omega_G(T_v) - \omega_G(T_{q,v}),\\
&= \sum_{v\in R_{H_Z}} \omega_G(T_v)- \omega_G(T_{q-1,v}) + \omega_G(T_{q-1,v}) - \omega_G(T_{q,v}),\\
&\geq \sum_{v\in R_{H_Z}} \omega_G(T_v)- \omega_G(T_{q-1,v})\\
&=\sum_{v\in R_{H_Z}} a_v(\Delta_{G_{\{1,\dots, d\}\setminus S} }
	(\Start_{G_{\{1,\dots, d\}\setminus S} } (H_Y)) +
	\delta_{H_Y}({\mathcal D}_{H_Z}(v)) - \Delta_{H_Z}(v))\\
&\geq \sum_{v\in R_{H_Z}} a_v(\delta_{H_Y}({\mathcal D}_{H_Z}(v)) - \Delta_{H_Z}(v))\\
	&= 0.
\end{align*}
	
{\it Case 2:} Assume $k\geq l$. 
The reader might find it helpful to recall Figure \ref{theorem 5.2 set ups} throughout this case.
Denote $F = \{\pi(d-2)\}$.
Note that $F$ consists of the ``most dominating direction" of $H$ and $J$ 
under the order $\mathcal{D}$.
We can write $z=(z_1,\dots, z_{d-2})$ and $y = (y_1,\dots, y_{d-2})$, and denote
\begin{align*}
	Z^* &= Z_2 \times \cdots Z_{\pi(d-1)-1} \times Z_{\pi(d-1)+1} \times \cdots \times Z_{d-1},\\ 
	Y^* &= Y_2 \times \cdots Y_{\pi(d-1)-1} \times Y_{\pi(d-1)+1} \times \cdots \times Y_{d-1},\\
	z^{\ast} &= (z_1,\dots, z_{\pi(d-2)-1},z_{\pi(d-2)+1},\dots, z_{d-2}),\\
	y^{\ast} &= (y_1,\dots, y_{\pi(d-2)-1},y_{\pi(d-2)+1},\dots, y_{d-2}).
\end{align*}
Further denote by $\mathcal{Y}$ and $\mathcal{Z}$ the domination orders on 
$H_{Y_F}(y^{\ast})$ and $H_{Z_F}(z^{\ast})$, respectively.
Note that the orders $\mathcal{Y}$ and $\mathcal{Z}$ are just the restrictions of
$\mathcal{D}_{H_Y}$ to $H_{Y_F}(y^{\ast})$ and $\mathcal{D}_{H_Z}$ to $H_{Z_F}(z^{\ast})$ respectively.
Denote $m = \mathcal{Y}(Y)$ and $n=\mathcal{Z}(z)$.
Note that $H_{Y_F}(y^{\ast})$ and $H_{Z_F}(z^{\ast})$ are one dimensional structures.
Furthermore, the graph $(H_{Y_F}(y^{\ast}), I_G(H_{Y_F}(y^{\ast}))$ is
isomorphic to the graph induced by the last part of the partition 
$\mathfrak{P}_{G_{\pi(d-2)+1}}$,
and the graph $(H_{Z_F}(z^{\ast}), I_G(H_{Z_F}(z^{\ast}))$ is isomorphic to the graph 
induced by the first part of the partition $\mathfrak{P}_{G_{\pi(d-2)+1}}$.
More formally,
\begin{align*}
	(H_{Y_F}(y^{\ast}), I_G(H_{Y_F}(y^{\ast}))
	&\cong
	(Y_{\pi(d-2)+1}, I_{G_{\pi(d-2)+1}}(Y_{\pi(d-2)+1}))\\
	(H_{Z_F}(z^{\ast}), I_G(H_{Z_F}(z^{\ast}))
	&\cong
	(Z_{\pi(d-2)+1}, I_{G_{\pi(d-2)+1}}(Z_{\pi(d-2)+1}))\\
\end{align*}
Thus, there are orders $\mathcal{\overline{Y}}$ and $\mathcal{\overline{Z}}$ in these
graphs, corresponding to the orders $\mathcal{Y}$ and $\mathcal{Z}$.
	
{\it Case 2.1:} Assume $m<n$ (see Figure \ref{Case2.1 Setup}).
	
\begin{figure}
\centering
\includegraphics[width=0.75\textwidth]{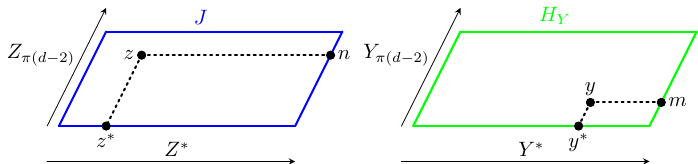}
\caption{General setup for case 2. Case 2.1 is shown, since $m<n$.}
\label{Case2.1 Setup}
\end{figure}
	
Denote
\begin{align*}
	y' &= (z_1,\dots , z_{\pi(d-2)-1}, \mathcal{\overline{Z}}^{-1}(m), z_{\pi(d-2)+1},\dots, , z_{d-2}),\\
	z' &= (y_1,\dots , y_{\pi(d-2)-1}, \mathcal{\overline{Y}}^{-1}(n), y_{\pi(d-2)+1},\dots, , y_{d-2}).
\end{align*}
Denote by $E$ the set of all two-dimensional sections under 
$\mathcal{D}_{H_Z}$ and $S=\{1,d\}$ of $\Stack_G(\psi)$ preceding 
$\Stack_G(\psi)\cap V_{G_S(y')}$, i.e.,
\begin{align*}
	E = \bigcup_{\substack{v\in V_{H_Z}\\ v<_{\mathcal{D}_{H_Z}} y' }} 
	(\Stack_G(\psi)\cap V_{G_S(v)}).
\end{align*}
For $v=(y_1,\dots, y_{\pi(d-2)-1}, u, y_{\pi(d-2)+1},\dots, , y_{d-2})\in R_{H_Y}$ 
with $R_{H_Y}=\{v\in V_{H_Y}\bigm|y \leq_{\mathcal{D}_{H_Y}} v \leq_{\mathcal{D}_{H_Y}} z'\}$ denote
\begin{align*}
	v' &= (z_1,\dots , z_{\pi(d-2)-1}, 
	\mathcal{\overline{Z}}^{-1}(\mathcal{\overline{Y}}(u)), 
	z_{\pi(d-2)+1},\dots, , z_{d-2}).
\end{align*}
For $v\in R_{H_Y}$ denote by $W_v$ the projection of 
$V_{G_S(v)}\cap \Stack_G(\alpha_{q-1})\cap A$ to 
$V_{G_S(v')}\cap \Stack_G(\psi)$, i.e.,
\begin{align*}
	W_v = \{(v_1,v',v_d) \bigm | (v_1,v,v_d)\in V_{G_S(v)}\cap \Stack_G(\alpha_{q-1})\cap A \},
\end{align*}
and put
\begin{align*}
	W = \bigcup_{v\in R_{H_Y}} W_v.
\end{align*}
Finally, define
\begin{align*}
	A' = (A\setminus \Slice_G(q-1))\cup E\cup W.
\end{align*}
Set $A'$ has the following properties:
\begin{enumerate}
\item $A'$ is compressed. 
The proof is similar to case 1 and is left to the reader.
\item For each $v\in V_{H_Z}$ the set $V_{G_S(v)}\cap A'\cap \Stack_G(\alpha_q)$ (see Figure \ref{Case2.1 Dimond} for the general picture)
is an initial segment of $\mathcal{BL}_{G_S(v)}^2$ restricted to 
$\Stack_G(\alpha_q)$, since $A$ is strongly compressed.
\item For each $v\in V_{H_Z}$ the set $V_{G_S(v)}\cap A'\cap \Stack_G(\psi)$  (see Figure \ref{Case2.1 Dimond} for the general picture)
is an initial segment of $\mathcal{BL}_{G_S(v)}^2$ restricted to $\Stack_G(\psi)$,
since $A$ is strongly compressed and by the definition of $W$.
\end{enumerate}

\begin{figure}
	\centering
	\includegraphics[width=0.85\textwidth]{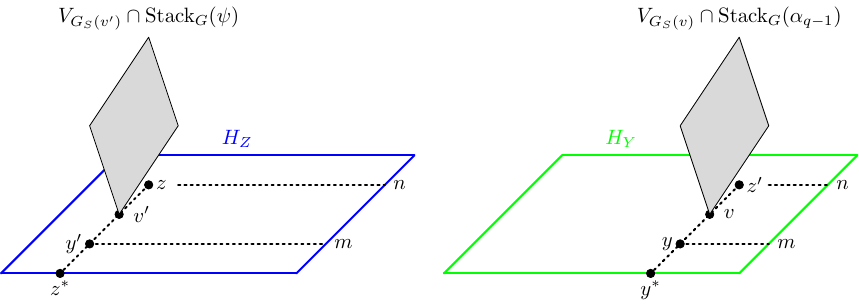}
	\caption{Setup for the pull part in Case 2.1}
	\label{Case2.1 Dimond}
\end{figure}

Denote $D' = \Comp_{G,S}(A')$. The set $D'$ can be constructed by moving 
$a_v$ vertices from $V_{G_S(v)}\cap A'\cap \Stack_G(\alpha_q)$ to 
$(V_{G_S(v)}\cap \Stack_G(\psi))\setminus A'$, for all $v\in R_{H_Z}$ where
$R_{H_Z}=\{v\in V_{H_Z}\bigm| y' \leq_{\mathcal{D}_{H_Z}} v \leq_{\mathcal{D}_{H_Z}} z\}$ and 
\begin{align*}
	a_v = \min \{|V_{G_S(v)}\cap A'\cap \Stack_G(\alpha_q)|,
	|(V_{G_S(v)}\cap \Stack_G(\psi))\setminus A'|\}.
\end{align*}
	
The set $D'$ can be obtained from $A'$ by applying the following steps for
every $v\in R_{H_Z}$: 
\begin{enumerate}
\item Remove the set $T_{q,v}$ consisting of the last $a_v$ vertices of 
$V_{G_S(v)}\cap A'\cap \Stack_G(\alpha_q)$ in the order $\mathcal{BL}_{G_S(v)}^2$ 
restricted to $\Stack_G(\alpha_q)$. 
\item Add the set $T_{q-1,v}$ consisting of the first $a_v$ vertices of 
$(V_{G_S(v)} \cap \Stack_G(\psi))\setminus A'$ 
in the order $\mathcal{BL}_{G_S(v)}^2$ restricted to $\Stack_G(\psi)$. 
\end{enumerate}
Corollary \ref{posetTrick} implies
\begin{align*}
	0 \leq |I_G(D')|-|I_G(A')| = 
	\sum_{v\in R_{H_Z}} \omega_G(T_{q-1,v})-\omega_G(T_{q,v}).
\end{align*}
For $v\in R_{H_Y}$ denote
\begin{align*}
	T_{v} = \{ &(v_1,v, v_d) \bigm | (v_1,v', v_d) \in T_{q-1,v'}\}.
\end{align*}
and consider the set
\begin{align*}
	D = \left( A \setminus \bigcup_{v\in R_{H_Z}} T_{q,v} \right) 
		\cup \bigcup_{v\in R_{H_Y}} T_v.
\end{align*}
Similarly to case 1 it can be shown that $D$ is compressed. One has,
\begin{align*}
	|I_G(D)| - |I_G(A)| &= 
	\sum_{v\in R_{H_Y}} \omega_G(T_v) - \omega_G(T_{q,v'}),\\
	&= \sum_{v\in R_{H_Y}} \omega_G(T_v)- \omega_G(T_{q-1,v'}) + 
	\omega_G(T_{q-1,v'}) - \omega_G(T_{q,v'}),\\
	&\geq \sum_{v\in R_{H_Y}} \omega_G(T_v)- \omega_G(T_{q-1,v'})\\
	&=\sum_{v\in R_{H_Y}} a_v(\Delta_{G_{\{1,\dots, d\}\setminus S} }
	(\Start_{G_{\{1,\dots, d\}\setminus S} } (H_Y)) +
	\Delta_{H_Y}(v) - \Delta_{H_Z}(v'))\\
	&\geq \sum_{v\in R_{H_Y}} a_v(0+\Delta_{H_Y}(v) - \Delta_{H_Z}(v'))\\
	&\geq 0.
\end{align*}
Note that the last inequality follows from the non-decreasing property
of $\mathfrak{P}_{G_2},\dots, \mathfrak{P}_{G_{d-1}}$.
	
{\it Case 2.2:} Now assume $m\geq n$.
Denote by $E$ the set of all sections under $\mathcal{D}_{H_Z}$ and $S$ of 
$\Stack_G(\psi)$ preceding $\Stack_G(\psi)\cap V_{G_S(z)}$, i.e.,
\begin{align*}
	E = \bigcup_{\substack{v\in V_{H_Z}\\ v<_{\mathcal{D}_{H_Z}} z }} 
	\Stack_G(\psi)\cap V_{G_S(v)}.
\end{align*}
Let $W$ be the projection of $V_{G_S(y)}\cap \Stack_G(\alpha_{q-1})\cap A$
to $V_{G_S(z)}\cap \Stack_G(\psi)$, i.e.,
\begin{align*}
	W = \{(v_1,z,v_d) \bigm | (v_1,y,v_d)\in V_{G_S(y)}\cap \Stack_G(\alpha_{q-1})\cap A \}.
\end{align*}
Consider the set
\begin{align*}
	A' = (A-\Slice_G(q-1))\cup E\cup W.
\end{align*}
Set $A'$ has the following properties:
\begin{enumerate}
\item $A'$ is compressed. 
The proof is similar to case 1 and is left to the reader.
\item The set $V_{G_S(z)}\cap A'\cap \Stack_G(\alpha_q)$ is an initial segment 
of $\mathcal{BL}_{G_S(z)}^2$ restricted to $\Stack_G(\alpha_q)$, 
since $A$ is strongly compressed.
\item The set $V_{G_S(z)}\cap A'\cap \Stack_G(\psi)$ is an initial segment 
of $\mathcal{BL}_{G_S(z)}^2$ restricted to $\Stack_G(\psi)$, since $A$ is 
strongly compressed and by the definition of $W$.
\end{enumerate}
	
Construct set $D' = \Comp_{G,S}(A')$ by moving $a$ vertices 
from $V_{G_S(z)}\cap A'\cap \Stack_G(\alpha_q)$ to 
$(V_{G_S(z)}\cap \Stack_G(\psi))\setminus A'$, where
\begin{align*}
	a = \min \{|V_{G_S(z)}\cap A'\cap \Stack_G(\alpha_q)|,
	|(V_{G_S(z)}\cap \Stack_G(\psi))\setminus A'|\}.
\end{align*}
	
More exactly, the set $D'$ can be obtained from $A'$ in two following
steps:
\begin{enumerate}
\item Remove the set $T_{q}$ consisting of the last $a$ vertices of 
$V_{G_S(z)}\cap A'\cap \Stack_G(\alpha_q)$ 
in the order $\mathcal{BL}_{G_S(z)}^2$ restricted to $\Stack_G(\alpha_q)$. 
\item Add the set $T_{q-1}$ consisting of first $a$ vertices of 
$(V_{G_S(z)} \cap \Stack_G(\psi))\setminus A'$ in the order $\mathcal{BL}_{G_S(z)}^2$ 
restricted to $\Stack_G(\psi)$. 
\end{enumerate}
Corollary \ref{posetTrick} implies
\begin{align*}
	0 \leq |I_G(D')|-|I_G(A')| = \omega_G(T_{q-1})-\omega_G(T_{q}).
\end{align*}
	
Denote
\begin{align*}
	T_{y} = \{ &(v_1,y, v_d) \bigm | (v_1,z, v_d) \in T_{q-1}\}
\end{align*}
and let
\begin{align*}
	D = \left( A \setminus T_{q} \right) \cup T_y.
\end{align*}
Similarly to case 1 it can be shown that the set $D$ is compressed,
One has:
\begin{align*}
	|I_G(D)| - |I_G(A)| &= \omega_G(T_y) - \omega_G(T_{q}),\\
&= \omega_G(T_y)- \omega_G(T_{q-1}) + \omega_G(T_{q-1}) - \omega_G(T_{q}),\\
	&\geq \omega_G(T_y)- \omega_G(T_{q-1})\\
	&= a(\Delta_{G_{\{1,\dots, d\}\setminus S} }
		(\Start_{G_{\{1,\dots, d\}\setminus S} } (H_Y)) +
		\Delta_{H_Y}(y) - \Delta_{H_Z}(z))\\
	&\geq a(0+\Delta_{H_Y}(y) - \Delta_{H_Z}(z))\\
	&\geq 0.
\end{align*}
The last inequality follows from the non-decreasing property
of $\mathfrak{P}_{G_2},\dots, \mathfrak{P}_{G_{d-1}}$.
	
In cases 1 and 2 we showed how move vertices from $\Slice_G(q)$ to 
$\Slice_G(q-1)$ to transform $A$ into set $D$ such that
$|A|=|D|$ and $|I_G(A)|\leq |I_G(D)|$. Make $D$ strongly compressed and
denote the resulting set by $E$.
By Theorem \ref{sliceCompression} there is a strongly compressed and sliced 
compressed set $F$ with
\begin{align*}
	|A|&=|D|=|E|=|F|,\\
	|I_G(A)| &\leq |I_G(D)|\leq |I_G(E)| \leq |I_G(F)|.
\end{align*}
Repeat the transformations described above until we get a stable set $B$. 
This will be the case because we replace some vertices with the ones that 
come earlier in the order $\mathcal{BL}_G^d$. The set $B$ is strongly 
compressed and block compressed, $|A|=|B|$, and $|I_G(A)|\leq |I_G(B)|$. 
\end{proof}

We are almost done with the proof of our main result.
Here is a summary of what we have established:
\begin{enumerate}
	\item Reduced the problem to strongly compressed sets.
	\item Reduced the problem to slice compressed sets.
	\item Reduced the problem to block compressed sets.
\end{enumerate}
However, it could be the case that we get a block compressed set which has only 
one partially filled block that is not ordered according to the order 
$\mathcal{BL}_G^d$.
But this is not a problem, since we assumed that the domination order 
on this block is optimal. In this case we just replace the set of vertices 
in this block with an initial segment of the same size in order $\mathcal{BL}_G^d$.
The resulting set will be optimal because of the second property of 
isoperimetric partitions.  Namely, for the partition
$\mathfrak{P}_{G_j} = 
\{{\mathcal O}_{G_j}[a_1,b_1] <_{{\mathcal O}_{G_j}} \cdots 
<_{{\mathcal O}_{G_j}} {\mathcal O}_{G_j}[a_k,b_k]\} $ one has
$|I({\mathcal O}_{G_j}[a_1,b_{i-1}], \{v\})| = \delta (a_i) =  
\Delta_{G_j} ({\mathcal O}_{G_j}(a_i))$ for all $v\in {\mathcal O}_{G_j}[a_i,b_i]$.
With this concluding remark the proof of the main result is complete.

%auto-ignore
\section{Applications of the Main result}

In this section we show that most previously known results on
edge-isoperimetric problems are corollaries of our main 
Theorem \ref{mainResult}.
Let us start with the theorem of Ahlswede-Cai as the most general one.
For this we define the \textit{atomic partition} of an isoperimetric graph 
$G$ with optimal order $\mathcal{O}$,
as $\mathfrak{A}_G = \{\{{\mathcal O}_G(1)\},\dots,\{{\mathcal O}_G(|V|)\}\}$ with 
$\{ {\mathcal O}_G(1)\} <_{{\mathcal O}_G}\dots <_{{\mathcal O}_G} \{{\mathcal O}_G(|V|)\}$.

\begin{corollary}\label{lgb_Atomic}
Let $G_1,\dots, G_d$ be isoperimetric graphs and for
$S\subset\{1,\dots,d\}$ let $\mathcal{BL}_{G_S}^{|S|}$ be the
block-lexicographic order based on atomic partitions 
$\mathfrak{A}_{G_1},\dots, \mathfrak{A}_{G_d}$.
If for all $i,j\in \{1,\dots, d\}$ with $i<j$ the order 
$\mathcal{BL}_{G_i\square G_j}^{2}$ is optimal,
then $\mathcal{BL}_G^d$ is optimal for $d\geq 3$.
\end{corollary}
\begin{proof}
Atomic partitions are regular and non-decreasing.
So, all the conditions of Theorem \ref{mainResult} are satisfied.
\end{proof}

The block-lexicographic order of Corollary \ref{lgb_Atomic}
is just the lexicographic order. Every block of this order consists of only one
vertex. The blocks are ordered according to their starts which, in
turn, are ordered lexicographically. Thus, there is a unique domination 
order of each block and unique block lexicographic order for each sub-product.
This way slices are $(d-1)$-dimensional sub-products and stacks are
$1$-dimensional ones. Also, the proofs of most statements in section 5 can be merely 
simplified and deduced from the properties of the strong compression.

Evidently, every application of Theorem \ref{AC_LGP} based on the
local-global principle follows from Corollary \ref{lgb_Atomic}. Let us 
mention just a few of them. In all these results the 2-dimensional case 
must be handled separately. In some cases, e.g., the hypercube it is 
rather trivial, whereas in other cases like cliques or Petersen graphs 
it requires more work. There are not any general methods known to us that 
eliminate multiple sub-cases by working with compressed sets in products of
two graphs.

\begin{corollary}\label{cubeCor}
The order $\mathcal{L}_{K^d_2}^d$ is optimal for every $d\geq 2$.
\end{corollary}

\begin{corollary}[Lindsey \cite{L_Cube}]\label{completeProd}
Suppose that $1\leq n_1\leq n_2\leq \cdots \leq n_d$.
The lexicographic order is optimal for $K_{n_1}\square \cdots \square K_{n_d}$.
\end{corollary}

\begin{corollary}[Bezrukov-Bulatovic-Kuzmanovski \cite{BBK}]
Suppose that $1\leq n_1\leq \cdots \leq n_d$ and consider the 
complete bipartite graphs $K_{n_1,n_1},\dots, K_{n_d,n_d}$.
The lexicographic order is optimal for 
$K_{n_1,n_1}\square \cdots \square K_{n_d, n_d}$.
\end{corollary}

\begin{corollary}[Bezrukov and Els\"asser \cite{BE_Regular} for $s=2$,
Bezrukov, Bulatovic and Kuzmanovski for arbitrary $s\geq 2$]\label{hspi}
Suppose that $s\geq 2$, $p\geq 3$ and $1\leq i \leq p-i$ and let
$G$ be an isoperimetric graph such that
\begin{align*}
	\delta_G= (
	\underline{0,\dots, p-1},
	\underline{p-i,\dots,p-i+(p-1) },\dots ,
	\underline{(s-1)(p-i),\dots, (s-i)(p-1)+p-1}).
\end{align*}
The lexicographic order is optimal for $G^d$, for any $d\geq 2$.
\end{corollary}

It is important to note that graphs with $\delta$-sequences specified
in Corollary \ref{hspi} exist. Such graphs are constructed in \cite{BBK} 
by a method involving the join operation on graphs. In case $s=2$ one can 
construct such a graph by removing $i$ disjoint perfect matchings form $K_{2p}$.
There are even more graphs studied in \cite{BE_Regular} and \cite{BBK} 
to which the local-global principle is applicable. 
Below we present several results that do not follow from the theorem of 
Ahlswede-Cai. For this we need to define the standard block-lexicographic 
order.

Let $G_1,\dots, G_d$ be isoperimetric graphs along with their 
standard partitions $\mathfrak{M}_{G_1},\dots,\mathfrak{M}_{G_d}$.
Recall from Theorem \ref{standardPartition} that the subgraphs induced by
monotonic sets are cliques. Hence, a subgraph of 
$G=G_1\square\cdots\square G_d$ induced by every block is isomorphic to 
the product of cliques which admits nested solutions according to the result 
of Lindsey (cf.  Corollary \ref{completeProd}). In fact, the optimal order 
is a domination order! We make it more explicit now.

Let $S\subseteq \{1,\dots d\}$ and $B= V_{K_{n_1}} \times \cdots \times
V_{K_{n_{|S|}}}$ be a block of $G_S$.
Define a total order $\eta: \{0,1\dots, |S|\}\rightarrow \{0,1,\dots |S|\}$
as follows: for $i,j\in \{1,\dots, |S|\}$ we say $i <_{\eta} j$ iff
$n_i<n_j$ or if $n_i=n_j$ and $i < j$.
Therefore, since $1\leq n_{\eta(1)}\leq \cdots \leq n_{\eta(d)}$
the lexicographic order is optimal for 
$K_{n_{\eta(1)}}\square \cdots \square K_{n_{\eta(d)}}$.
Note that $\eta \in \mathfrak{S}_d$
and there is a unique domination order $\mathcal{D}^{\eta, |S|}$.
So, for $K=K_{n_1}\square \cdots \square K_{n_{|S|}}$
we have the induced domination order $\mathcal{D}_K^{\eta, |S|}$ on $K$.
Furthermore, the order $\mathcal{D}_K^{\eta, |S|}$ is optimal,
since lexicographic order is optimal for 
$K_{n_{\eta(1)}}\square \cdots \square K_{n_{\eta(d)}}$.

Hence, for any block $B$ we can construct a permutation $\eta$ for which there is an optimal domination order on the graph
induced by a block.
It is easily seen that $\mathfrak{M}_{G_1}, \dots, \mathfrak{M}_{G_d}$
is a domination collection with these domination orders.
We call $\mathcal{D}^{\eta, |S|}_K$ the {\it standard block domination order}
on $B$.
Furthermore, we call the block lexicographic orders formed by 
the standard partitions $\mathfrak{M}_{G_1}, \dots, \mathfrak{M}_{G_d}$
and the standard block domination orders,
the {\it standard block lexicographic orders }
and for $S\subseteq \{1,\dots, d\}$ denote them by $\mathcal{SBL}_{G_S}^{|S|}$. 

Let us talk on the regularity of partitions for regular graphs.
Many authors in the past have noticed the following result.
\begin{lemma}\label{ComplementOptimal}
If $G=(V,E)$ is regular then $A\subseteq V$ is optimal iff $V\setminus A$ 
is optimal.
\end{lemma}

The above lemma provides a relationship between regular graphs and their
regular partitions.
\begin{corollary}\label{standardPartIsRegular}
Let $G=(V,E)$ be a regular isoperimetric graph with optimal order 
${\mathcal O}_G$ and let ${\mathcal O}_G'$ be its reverse order. 
Then the order ${\mathcal O}_G'$ is optimal and the partition $\mathfrak{M}_G$ 
is regular.
\end{corollary}
\begin{proof}
Note that for any $i\in \{1,\dots, V\}$ the set ${\mathcal O}_G[1,i]$ is optimal.
Thus, ${\mathcal O}_G'[1, |V|-i+1]$ is optimal for all $i\in \{1,\dots, |V|\}$
by Lemma \ref{ComplementOptimal}.
Therefore, the first and last monotonic segments of the order ${\mathcal O}_G$ 
are of the same size.
\end{proof}

By Corollary \ref{standardPartIsRegular}, if the graph $G$ is isoperimetric 
and regular, then $\mathfrak{M}_G$ is an isoperimetric, non-decreasing and 
regular partition. See \cite{BS} for properties of the $\delta$-sequences of 
regular graphs.

\begin{corollary}\label{mainResultCor}
Let $G_1,\dots, G_d$ be regular isoperimetric graphs with their respective
standard partitions $\mathfrak{M}_{G_1},\dots,\mathfrak{M}_{G_d}$ and
$G=G_1\square \cdots \square G_d$.
If $\mathcal{SBL}_{G_i\square G_j}$ is optimal for all $i<j$ then 
$\mathcal{SBL}_G^d$ is optimal for $d\geq 3$.
\end{corollary}
\begin{proof}
We have that $\mathfrak{M}_{G_i}$ is non-decreasing and regular.
Thus, $\mathfrak{M}_{G_1}, \dots, \mathfrak{M}_{G_d}$
is a regular domination collection by the definition of the standard 
block-domination order. The statement then follows from 
Theorem \ref{mainResult}.
\end{proof}

Taking all this into account, we show how to obtain some results which are
not covered by the theorem of Ahlswede-Cai.

\begin{corollary}[Bezrukov-Das-Elsässer \cite{BDE_Petersen}]\label{petersenCor}
Let $G_1$ be the Petersen graph, $G_2=K_2$, and $d_1,d_2\geq 0$. Then
$G=G_1^{d_1} \square G_2^{d_2}$ has nested solutions and
the standard block-lexicographic order is optimal.
\end{corollary}
\begin{proof}
One just needs to show that the standard block-lexicographic order is 
optimal for $G_1^2$, $G_1\square G_2$ and $G_2^2$.
This is trivial for $G_2^2$ and the other cases are covered in Chapter 9 of 
\cite{H_Book}, and in \cite{BDE_Petersen}.
Therefore, the statement follows from Corollary \ref{mainResultCor}.
\end{proof}

The authors of \cite{BDE_Petersen} first proved Corollary \ref{petersenCor} 
for $d_1\geq 0$ and $d_2=0$. 
The general case required more work.
The authors had to define a new order and handle multiple new cases.
In the next results the 
two-dimensional cases needed for Corollary \ref{mainResultCor} are covered
in \cite{C_Tori}. 

\begin{corollary}[Carlson \cite{C_Tori}]\label{tori1}
If $d_1,d_2,d_3\geq 0$ then the standard block-lexicographic order is 
optimal for $C_5^{d_1}\square C_4^{d_2}\square C_3^{d_3}$.
\end{corollary}

\begin{corollary}[Carlson \cite{C_Tori}]\label{tori2}
If $d_1,d_2,d_3\geq 0$ and $n\geq 6$
then the standard block-lexicographic order is optimal for 
$C_n\square C_5^{d_1}\square C_4^{d_2}\square C_3^{d_3}$.
\end{corollary}

It is worse to note that in \cite{C_Tori} Carlson
first proved Theorem \ref{C_Cycle}. Then after several auxiliary results 
and multiple statements about higher dimensions he finally established the
results outlined in Corollaries \ref{tori1} and \ref{tori2}.
Throughout these steps Carlson needed to define some special orders.
In contrast to this, in our case we need just one standard 
block-lexicographic order.
A similar situation occurred in \cite{BDE_Petersen} by proving 
Corollary \ref{petersenCor}.
Our methods allow us to further extend the results of Carlson.

\begin{corollary}
If $d_1,d_2,d_3,d_4\geq 0$ and $n\geq 6$
then the standard block-lexicographic order is optimal for 
$C_5^{d_1}\square C_4^{d_2}\square K_2^{d_3}\square C_3^{d_4}$ and
$C_n\square C_5^{d_1}\square C_4^{d_2}\square K_2^{d_3}\square C_3^{d_4}$.
\end{corollary}
\begin{proof}
We just need to check the two-dimensional cases of Corollary
\ref{mainResultCor} for the product of every graph in question with $K_2$.
They, however, follow from the two-dimensional cases of \cite{C_Tori} 
%with the other graphs,
since an initial segment of size $2$ in any connected graph induces $K_2$.
Also, note that $C_3$ is isomorphic to $K_3$.
Therefore, the statement follows from \ref{mainResultCor}.
\end{proof}

%auto-ignore
\section{Concluding remarks and future directions}
In this paper we generalized the Ahlswede-Cai Theorem
and proved that almost all results in the area of edge-isoperimetric problems
on Cartesian products of graphs are consequences of our main result.
In turn, our result can be generalized like it is done in the original
local-global principle paper \cite{AC_LGB, AC_SMF}.
In particular, Theorem \ref{mainResult} has an analog for $\Theta_G$.
In general, there is an analog of Theorem \ref{mainResult}
for any sub-modular function defined in \cite{AC_LGB,AC_SMF}.
We only worked with $I_G$ to make the paper easier to read.

Block-lexicographic orders are made up of two parts:
the block part and the lexicographic part.
One can think of defining block-domination orders
similarly to how we handled domination orders
by starting with lexicographic order.
A natural question is if we can go even further.
If we have some $d$-dimensional order $\mathcal{O}$,
can we say anything on the optimality of a block order where the blocks are
ordered according to the order $\mathcal{O}$?
What properties does $\mathcal{O}$ have to satisfy?
Note that the optimal orders in Theorems \ref{BL_Grid} and \ref{AB_Trees}
are block orders, but not lexicographic ones.
The orders use the standard partitions for the blocks.
We would be surprised if the local-global principles could not be extended to
such orders. Answering these questions would settle Harper's question 
completely. Here is a conjecture that we believe would help to find answers
to the posed questions.

\begin{conjecture}
If $d_1,d_2\geq 0$ and $n_1,n_2\geq 2$ 
then $P_{n_1}^{d_1}\square K_{n_2}^{d_2}$ has nested solutions.
\end{conjecture}

A lot of work has been done for the products of 3 or more graphs and very
powerful methods have been developed. However, the 2-dimensional case still 
requires a special treatment and no general methods are known for it. 
One possible direction could be to improve the upper bound 
for $i$ in Corollary \ref{hspi}. Examples show that it can be much higher. For
$s=2$ the following conjecture is proved in \cite{BE_Regular}.

\begin{conjecture}
Statement of Corollary \ref{hspi} is valid for $s\geq 2$, $p\geq 3$ and
$1\leq i \leq p - p/s$. For $i> p - p/s$ there are no nested solutions.
\end{conjecture}

The optimal order for the product of two Petersen graphs was
established by using a computer in \cite{BDE_Petersen}.
Harper developed a method in \cite{H_Book} which in many cases can be 
used without computers.
It would be interesting to develop general methods for the products of two
graphs to handle cases like the powers of Petersen graph.
To stimulate this we present another conjecture.

\begin{conjecture}
If $d_1,d_2,d_3,d_4d_5\geq 0$, $n\geq 6$ and $G$ is the Petersen graph,
then the standard block lexicographic order is optimal for 
$C_5^{d_1}\square G^{d_2}\square C_4^{d_3}\square K_2^{d_4}\square C_3^{d_5}$
and
$C_n\square C_5^{d_1}\square G^{d_2}\square C_4^{d_3}\square K_2^{d_4}\square C_3^{d_5}$.
\end{conjecture}

\bibliography{pull-push}

\end{document}